\documentclass[11pt ,twoside, leqno]{article}

\usepackage{mathrsfs}
\usepackage{amssymb}
\usepackage{amsmath}
\usepackage{amsthm}

\usepackage{amsfonts}
\usepackage[colorlinks=true, pdfstartview=FitV, linkcolor=blue, citecolor=blue, urlcolor=blue]{hyperref}
\usepackage{color}

\usepackage{ esint }

\allowdisplaybreaks

\renewcommand\hat{\widehat}

\def\supp{{\rm{\,supp\,}}}

\def\bint{{\ifinner\rlap{\bf\kern.30em--}
\int\else\rlap{\bf\kern.35em--}\int\fi}\ignorespaces}

\def\sbint{{\ifinner\rlap{\bf\kern.32em--}
\hspace{0.078cm}\int\else\rlap{\bf\kern.45em--}\int\fi}\ignorespaces}

\newtheorem{theorem}{Theorem}[section]
\newtheorem{lemma}[theorem]{Lemma}
\newtheorem{corollary}[theorem]{Corollary}
\newtheorem{proposition}[theorem]{Proposition}

\theoremstyle{definition}
\newtheorem{remark}[theorem]{Remark}
\newtheorem{definition}[theorem]{Definition}
\numberwithin{equation}{section}

\numberwithin{equation}{section}

\begin{document}

\title{\Large\bf Campanato spaces via quantum Markov semigroups on finite von Neumann algebras \footnotetext{\hspace{-0.35cm} {\it 2020 Mathematics Subject Classification}.
{Primary  46L52; Secondary 30H99, 47D07.}
\endgraf{\it Key words and phrases.}  Campanato spaces, Lipschitz spaces, quantum Markov semigroups, von Neumann algebras.
}}

\author{Guixiang Hong, Yuanyuan Jing}
\date{  }

\maketitle
\vspace{-0.8cm}

\begin{center}
\begin{minipage}{13cm}\small
{\noindent{\bf Abstract} \
We study the Campanato spaces associated with quantum Markov semigroups on a finite von Neumann algebra $\mathcal M$.
 Let $\mathcal T=(T_{t})_{t\geq0}$ be a Markov semigroup, $\mathcal P=(P_{t})_{t\geq0}$ the subordinated Poisson semigroup and
 $\alpha>0$. The column Campanato space ${\mathcal{L}^{c}_{\alpha}(\mathcal{P})}$ associated to $\mathcal P$ is defined to be the subset of $\mathcal M$ with finite norm which is given by
\begin{align*} \|f\|_{\mathcal{L}^{c}_{\alpha}(\mathcal{P})}=\left\|f\right\|_{\infty}+\sup_{t>0}\frac{1}{t^{\alpha}}\left\|P_{t}|(I-P_{t})^{[\alpha]+1}f|^{2}\right\|^{\frac{1}{2}}_{\infty}.
\end{align*}
The row space ${\mathcal{L}^{r}_{\alpha}(\mathcal{P})}$ is defined in a canonical way. In this article, we will first show the surprising coincidence of these two spaces ${\mathcal{L}^{c}_{\alpha}(\mathcal{P})}$ and ${\mathcal{L}^{r}_{\alpha}(\mathcal{P})}$ for $0<\alpha<2$. This equivalence of column and row norms is generally unexpected in the noncommutative setting. The approach is to identify both of them as the Lipschitz space ${\Lambda_{\alpha}(\mathcal{P})}$.
This coincidence passes to the little Campanato spaces $\ell^{c}_{\alpha}(\mathcal{P})$ and $\ell^{r}_{\alpha}(\mathcal{P})$ for $0<\alpha<\frac{1}{2}$ under the condition $\Gamma^{2}\geq0$.
 We also show that any element in ${\mathcal{L}^{c}_{\alpha}(\mathcal{P})}$ enjoys the higher order cancellation property, that is, the index $[\alpha]+1$ in the definition of the Campanato norm can be replaced by any integer greater than $\alpha$. It is a surprise that this property holds without further condition on the semigroup. Lastly, following Mei's work on BMO, we also introduce the spaces ${\mathcal{L}^{c}_{\alpha}(\mathcal{T})}$ and explore their connection with ${\mathcal{L}^{c}_{\alpha}(\mathcal{P})}$. All the above-mentioned results seem new even in the (semi-)commutative case.}
\end{minipage}
\end{center}

\section{Introduction}
\hskip\parindent
The Campanato spaces, introduced by Campanato \cite{c63}, play an important role in the study of classic analysis and have been studied extensively. We refer the interested reader to e.g. \cite{c64,m17,p69,rss13,tw80} for developments and applications.  Let us first recall the definition of the Campanato spaces in $\mathbb {R}^{n}$. Let $\alpha\geq0$ and $[\alpha]$ be the integer part of $\alpha$, the classical Campanato space is defined as a subspace of $L_{\mathrm{loc}}^{2}(\mathbb{R}^{n})$ with finite norm which is given by
\begin{align}\label{02}
\|f\|_{\mathcal{L}_{\alpha}(\mathbb{R}^{n})}=\sup_{B}\frac{1}{|B|^{\frac{\alpha}{n}}}\left(\frac{1}{|B|}\int_{B}|f(x)-P_{B}f(x)|^{2}dx\right)^{\frac{1}{2}},
\end{align}
where the supremum is taken over all balls $B$ in $\mathbb{R}^{n}$ and $P_{B}f$ is the polynomial of degree at most $[\alpha]$.
As noted in \cite{g85}, for $\alpha=0$, this coincides with the classical BMO space introduced by John and Nirenberg in \cite{jn61}; and for $\alpha>0$, these spaces are variants of the homogeneous Lipschitz spaces $\Lambda_{\alpha}(\mathbb{R}^{n})$ which are duals of certain Hardy spaces, see also \cite{jtw83,tw80,l08}.

In the case of $\alpha=0$, the BMO spaces defined purely via semigroups were introduced by Mei \cite{m08,m12} with the main aim to develop BMO theory in the noncommutative setting where the metric or geometric structure is not available. We refer the reader also to e.g. \cite{sv74,v80,xy05,XY05} for BMO theory that in turn partly motivated Mei's work. In these two papers, Mei studied the Fefferman-Stein dualities. Later on, in \cite{mj12} the authors introduced various BMO spaces associated to a semigroup and investigated their connection and interpolation. See also \cite{tms19} for more results on functional calculus on the semigroup BMO spaces.


In the present paper we are interested in the noncommutative theory of Campanato spaces in the remaining cases $\alpha>0$.

Let $\mathcal{M}$ be a von Neumann algebra with a normal faithful  finite  trace $\tau$. Let $\mathcal T=(T_{t})_{t\geq0}$ be a quantum semigroup acting on $\mathcal{M}$, $\mathcal P=(P_{t})_{t\geq0}$ its subordinated Poisson semigroup and $\alpha>0$. Motivated by the BMO space introduced in \cite{m08}, the column Campanato space is defined as
 \begin{align*}
   \mathcal{L}_{\alpha}^{c}(\mathcal P)=\left\{f\in\mathcal{M}: \,\, \left\|f\right\|_{\mathcal{L}_{\alpha}^{c}(\mathcal P)}<\infty\right\},
 \end{align*}
where
 \begin{align}\label{p1}
   \|f\|_{\mathcal{L}^{c}_{\alpha}(\mathcal{P})}= \|f\|_{\infty}+\sup_{t>0}\frac{1}{t^{\alpha}}\left\|P_{t}|(I-P_{t})^{[\alpha]+1}f|^{2}\right\|^{\frac{1}{2}}_{\infty}.
 \end{align}
Due to the noncommutativity, one defines also the row norm  $\left\|f\right\|_{\mathcal L_{\alpha}^{r}(\mathcal P)}=\left\|f^{*}\right\|_{\mathcal L_{\alpha}^{c}(\mathcal P)}$ and similarly the row space $ \mathcal{L}_{\alpha}^{r}(\mathcal P)$. Then the mixture space $  \mathcal{L}_{\alpha}^{cr}(\mathcal P)$ is defined as $\mathcal{L}_{\alpha}^{c}(\mathcal P)\cap \mathcal{L}_{\alpha}^{r}(\mathcal P)$ equipped with the intersection norm.
For $0<\alpha<1$, one may also define the little column Campanato norm: for $f\in\mathcal M$,
 \begin{align}\label{p3}
    \|f\|_{\ell^{c}_{\alpha}(\mathcal{P})}= \|f\|_{\infty}+\sup_{t>0}\frac{1}{t^{\alpha}}\left\|P_{t}|f|^{2}-|P_{t}f|^{2}\right\|^{\frac{1}{2}}_{\infty},
 \end{align}
and $\|f\|_{\ell_{\alpha}^{cr}(\mathcal{P})}=\max\left\{ \|f\|_{\ell^{c}_{\alpha}(\mathcal{P})},\|f\|_{\ell^{r}_{\alpha}(\mathcal{P})}\right\}$, where the row norm is given by $\|f\|_{\ell^{r}_{\alpha}(\mathcal{P})}=\|f^{*}\|_{\ell^{c}_{\alpha}(\mathcal{P})}$. The resulting spaces ${\ell^{c}_{\alpha}(\mathcal{P})}$, ${\ell^{r}_{\alpha}(\mathcal{P})}$ and ${\ell^{cr}_{\alpha}(\mathcal{P})}$ are defined in the standard way.

To the best of our knowledge, the above definitions of the Campanato spaces involving only semigroups are new even in the commutative setting while there have appeared some definitions via semigroups combined with geometric objects like balls in the literature such as \cite{ddy05, dxy07, dy09}. When $(P_{t})_{t\geq0}$ is the classical Poisson semigroup $(e^{-t\sqrt{-\Delta}})_{t\geq0}$ on $\mathbb R^n$, like in the case $\alpha=0$ (cf. \cite{XY05}), one may check that both \eqref{p1} and \eqref{p3} are equivalent to the classical inhomogeneous Campanato norm with the homogeneous part given by \eqref{02} for $0<\alpha<1/2$, see also \cite{ddy05, dy09}. For large $\alpha$'s, we believe that they are still equivalent but it needs more efforts. And one of our main results---Theorem \ref{e25} below, showing for $0<\alpha<2$ the coincidence of  \eqref{p1} and the classical inhomogeneous Lipschitz norm \eqref{l1} which is in turn well-known to be a variant of the classical inhomogeneous Campanato norm (cf. \cite{jtw83}), will provide positive evidence to this conjecture.


A semigroup $(T_{t})_{t\geq0}$ is called a Markov semigroup if it is unital, completely positive, trace-preserving and symmetric, which will be recalled in detail in Section 2.2.

\begin{theorem}\label{b3}
Let $(T_{t})_{t\geq0}$ be a quantum Markov semigroup, and $(P_{t})_{t\geq0}$ the Poisson semigroup subordinated to $(T_{t})_{t\geq0}$.
\begin{itemize}
  \item [\rm{(i)}]For any $0<\alpha<2$, we have the following isomorphisms with equivalent norms
\begin{align*}
 \mathcal L_{\alpha}^{c}(\mathcal P)=\mathcal L_{\alpha}^{r}(\mathcal P)=\mathcal L_{\alpha}^{cr}(\mathcal P).
\end{align*}
  \item [\rm{(ii)}] For $0<\alpha<\frac{1}{2}$, under the condition $\Gamma^{2}\geq0$ we have $$\ell^{c}_{\alpha}(\mathcal{P})=\ell^{r}_{\alpha}(\mathcal{P})=\ell_{\alpha}^{cr}(\mathcal{P})$$
  with equivalent norms. Here the notation $\Gamma^2$ means the Bakry-\'Emery's iterated gradient form that will be recalled in Section 6.
\end{itemize}
\end{theorem}
Note that due to the noncommutativity, it is well-known that the column space is usually not isomorphic to the row space such as in the cases of Hardy spaces or BMO spaces \cite{Mei07}, which is also one of the key sources of difficulty in noncommutative analysis. Thus Theorem \ref{b3} presents a new phenomenon. Moreover, Theorem \ref{b3}(i) implies that
the space $\mathcal L_{\alpha}^{c}(\mathcal P)$ is completely isomorphic to $\mathcal L_{\alpha}^{r}(\mathcal P)$ when the operator space structure is provided by the following identification
$$\|f\|_{M_{n}( \mathcal L_{\alpha}^{c}(\mathcal P))}:=\|f\|_{\mathcal L_{\alpha}^{c}((P_t\otimes I_n)_{t\geq0})},\quad\|f\|_{M_{n}( \mathcal L_{\alpha}^{r}(\mathcal P))}:=\|f\|_{( \mathcal L_{\alpha}^{r}((P_t\otimes I_n)_{t\geq0})}$$
where $I_n$ is the identity map on $M_n$---the algebra of $n\times n$ complex matrices.

Theorem \ref{b3}(ii) will be a consequence of Theorem \ref{b3}(i) once one can identify the Campanato spaces with the little ones. This identification will involve $\Gamma^2\geq0$, which is motivated by Junge and Mei's work \cite{mj12}. See Section 6 for details. The difficult part lies on Theorem \ref{b3}(i). We will establish this isomorphism by identifying both the column and row Campanato spaces with the Lipschitz spaces associated with quantum Markov semigroups.

For a given quantum Markov semigroup $(T_t)_{t\geq0}$ and its subordinated Poisson semigroup $(P_t)_{t\geq0}$ and $\alpha>0$, the Lipschitz space ${\Lambda_{\alpha}(\mathcal{P})}$ is the subset of $\mathcal M$ with finite norm given by
\begin{align}\label{l1}
 \|f\|_{\Lambda_{\alpha}(\mathcal{P})}=\|f\|_{\infty}+\sup_{t>0}\frac{1}{t^{\alpha-([\alpha]+1)}}\left\|\frac{\partial^{[\alpha]+1}P_{t}f}{\partial t^{[\alpha]+1}}\right\|_{\infty}.
\end{align}
The above definition can date back to Stein \cite{s70} in 1970 for the Poisson semigroup on $\mathbb R^n$. There seems very little work in the literature for other semigroups or the quantum ones, and most of the Lipschitz spaces existing in the literature are defined through differences, see e.g. \cite{w96,w98,xxy18} for the ones on quantum tori.
\begin{theorem}\label{e25}
Suppose $f\in \mathcal{M}$. Then for $\alpha>0$, we have
\begin{align}\label{b2}
 \left\|f\right\|_{\mathcal L_{\alpha}^{c}(\mathcal P)}\lesssim_{\alpha}\left\|f\right\|_{\Lambda_{\alpha}(\mathcal P)}.
\end{align}
Furthermore, if $0<\alpha<2$, we have
\begin{align}\label{192}
 \left\|f\right\|_{\Lambda_{\alpha}(\mathcal P)}\lesssim_{\alpha}\left\|f\right\|_{\mathcal L_{\alpha}^{c}(\mathcal P)}.
\end{align}
Thus, $\mathcal L_{\alpha}^{c}(\mathcal P)=\Lambda_{\alpha}({\mathcal P})$ with equivalent norms when $0<\alpha<2$. Similar statements hold for the row Campanato spaces.
\end{theorem}


Theorem \ref{e25} not only plays a key role in showing our main result---Theorem \ref{b3}(i), but also for the first time establishes the equivalence relationship between the semigroup Campanato spaces and the semigroup Lipschitz spaces. This sheds new light on the commutative theory, for which one may see e.g. \cite{g85,jtw83, tw80, c64,t92,g83}. In these references, it is proved  on $\mathbb R^n$ that the coincidence of the Campanato spaces defined geometrically such as \eqref{02} and the Lipschitz spaces defined by difference.

\bigskip
In handling the BMO estimate of some spectral multipliers associated with operators of various types such as the H\"ormander-type spectral multipliers \cite{DSY} and the Schr\"odinger groups \cite{{cdly20}},  a higher order cancellation property of BMO spaces \cite{hm09, sgd11} plays an essential role. See also a recent work \cite{FHW} for some sharp estimates of the Schr\"odinger groups and a higher order cancellation BMO in the noncommutative setting.
We show with a surprise that the Campanato spaces defined via Markov semigroup admit this self-improving property automatically, that is, the integer $[\alpha]+1$ in the definition of $\mathcal{L}^{c}_{\alpha}(\mathcal{P})$-norm---\eqref{p1}---can be substituted for any integer greater than $\alpha$.

Fix $\alpha>0$ and an integer $k\geq[\alpha]+1$. We define for $f\in\mathcal M$,
\begin{align}\label{09}
  \left\|f\right\|_{\mathcal L_{\alpha,k}^{c}(\mathcal P)}=\|f\|_{\infty}+\sup_{t>0}\frac{1}{t^{\alpha}}
   \left\|P_{t}|(I-P_{t})^{k}f|^{2}\right\|_{\infty}^{\frac{1}{2}},
   \end{align}
   and the resulting space ${\mathcal L_{\alpha,k}^{c}(\mathcal P)}$ in an obvious way.

\begin{theorem}\label{22}
Let $0<\alpha<\alpha_{0}$, where $\alpha_{0}$ verifies $2^{\alpha_{0}-2}+3^{\alpha_{0}-3}=1$. Then, the space $\mathcal L_{\alpha}^{c}(\mathcal P)$ coincides with $\mathcal L_{\alpha,k}^{c}(\mathcal P)$ with equivalent norms for any integer $k\geq[\alpha]+1$.
\end{theorem}

Our approach to this equivalence is again via the Lipschitz spaces. For this purpose, we show that the semigroup Lipschitz spaces admit a higher order cancellation property, see Section 3 for details. As justified above, Theorem \ref{22} and its proof are completely new in the commutative setting.


Finally, motivated by the work \cite{tms19,mj12} on the semigroup BMO spaces, we will introduce the Campanato/Lipschitz spaces associated to the semigroup $(T_t)_{t\geq0}$ itself, $\mathcal{L}^{c}_{\alpha}(\mathcal{T})$, $\ell_{\alpha}^{c}(\mathcal T)$ and ${\Lambda_{\frac{\alpha}{2}}(\mathcal T)}$, and explore their connection with the ones via the subordinated Poisson semigroup $\mathcal{L}^{c}_{\alpha}(\mathcal{P})$, $\ell_{\alpha}^{c}(\mathcal P)$ and ${\Lambda_{\alpha}(\mathcal P)}$. To this end, we need the quasi-monotone property of semigroups that will be recalled in Section 2.2.

\begin{theorem}\label{q2}
Let $(T_{t})_{t\geq0}$ be a Markov semigroup, $(P_{t})_{t\geq0}$ its subordinated Poisson semigroup. Assume that $(T_{t})_{t\geq0}$ is quasi-monotone.  Let $f\in \mathcal M$ and $0<\alpha<1$. Then
\begin{itemize}
  \item [\rm{(i)}]  $\|f\|_{\mathcal L_{\alpha}^{c}(\mathcal P)}\lesssim \|f\|_{\mathcal L_{\frac{\alpha}{2}}^{c}(\mathcal T)}$,
  \item [\rm{(ii)}]  $\|f\|_{\Lambda_{\alpha}(\mathcal P)}\lesssim\|f\|_{\Lambda_{\frac{\alpha}{2}}(\mathcal T)}$,
  \item [\rm{(iii)}] and $\|f\|_{\ell_{\frac{\alpha}{2}}^{c}(\mathcal T)}\lesssim \|f\|_{\ell_{\alpha}^{c}(\mathcal P)}$ if moreover $(T_{t})_{t\geq0}$ has the $\Gamma^{2}\geq0$ property.
\end{itemize}
\end{theorem}

\begin{remark}
All the results in the present paper are restricted to finite von Neumann algebras, this is mainly due to the technical reasons, and we believe that they admit generalizations on semifinite von Neumann algebras. On the other hand, in Theorem \ref{b3}(i), Theorem \ref{e25} and Theorem \ref{22}, there are some restrictions on the scale of $\alpha$ because of technical difficulties, and we also believe that they should hold for all $\alpha>0$. All this will be taken care of elsewhere.
\end{remark}

\begin{remark}
As the reader might have noted that only the inhomogeneous space is considered above, that is the presence of $\|\cdot\|_\infty$ in the definition of the norm in \eqref{p1}. The reason behind is that in noncommutative setting we have no idea how to define the homogeneous Campanato/Lipschitz spaces. Indeed, in \eqref{02}, we refer to the definition of the classical homogeneous Campanato space, which can be defined as a subspace of the ambient space $L_{\mathrm{loc}}^{2}(\mathbb{R}^{n})$. However, at the time of writing, we are unable to find a proper ambient space for homogeneous Campanato spaces in the noncommutative setting. Actually, it is an open problem to find suitable noncommutative analogues of local $L_{p}$ spaces or (tempered) distributions such that one may define BMO and Campanato spaces as subspaces. Inspired by Stein's definition of inhomogeneous Lipschitz spaces (cf. \cite{s70}), we therefore investigate inhomogeneous Campanato spaces as given in \eqref{p1} which can be well defined inside $L_\infty$. 

On the other hand, if the homogeneous Campanato/Lipschitz spaces were well-defined, then all the results in the present paper hold also true for these spaces since all the proof appearing in the present paper are concerned with only the homogeneous part of the norms. 
\end{remark}

The rest of the paper follows essentially the order presented in this introduction.


Throughout the paper, we will occasionally use the notation $A\lesssim B$ to indicate that $A\leq CB$ where $C$ is a positive constant which is independent of the main parameters, and the notation $A\lesssim_{\alpha}B$ to denote $A\leq C_{\alpha}B$ where $C_{\alpha}$ is a positive constant depending only  on the parameter $\alpha$. 

\smallskip

{\it Note}: After finishing the preliminary version of this paper, we learned that a slightly different {\it homogeneous} Lipschitz space defined through quantum Markov semigroup had been studied in thorough by  Gonz$\mathrm{\acute{a}}$lez-P$\mathrm{\acute{e}}$rez in \cite{g19}. In particular
he introduced the  {\it homogeneous} semigroup Campanato spaces and proved the easy part of Theorem \ref{e25} in the case $0<\alpha<1$, and left the difficult one \eqref{192} as an open problem (see the end of \cite{g19}). 


\section{Preliminary \label{s2}}
\hskip\parindent
In this section, we present the definitions and some basic properties about von Neumann algebras and Markov semigroups of operators that we will need throughout the paper.
\subsection{von Neumann algebras}
\hskip\parindent
Let $\mathcal{M}$ be a von Neumann algebra and $\mathcal{M}_{+}$ be its positive part. We recall that a trace on $\mathcal{M}$ is a map $\tau:\mathcal{M}_{+}\rightarrow[0,\infty)$ satisfying $\tau(x+y)=\tau(x)+\tau(y)$  for any $x,y\in\mathcal{M}_{+}$;  $\tau(\lambda x)=\lambda\tau(x)$ for any $\lambda\in[0,\infty)$, $x\in\mathcal{M}_{+}$;  $\tau(x^{*}x)=\tau(xx^{*})$ for any $x\in\mathcal{M}$. We will say $\tau$ is
\begin{itemize}
  \item   faithful, if $\tau(x)=0$ implies $x=0$;
  \item  normal, if $\sup_{\alpha}\tau(x_{\alpha})=\tau(\sup_{\alpha}x_{\alpha})$ for any bounded increasing net $x_{\alpha}$ in $\mathcal{M}_{+}$;
  \item  semifinite, if for any $x\in\mathcal{M}_{+}$, there is a $y\in\mathcal{M}_{+}$ such that $0<y\leq x$ with $\tau(y)<\infty$;
  \item finite, if $\tau(1)<\infty$ (1 denotes the identity of $\mathcal{M}$).
\end{itemize}
 A von Neumann algebra $\mathcal{M}$ is called finite if it admits a normal finite faithful trace.

Now, let us briefly review the noncommutative $L^{p}$-spaces. Let $\mathcal{M}$ be a von Neumann algebra equipped with a normal semifinite faithful trace $\tau$ and $S^{+}_{\mathcal{M}}$ be the set of all positive elements $x\in\mathcal{M}$ with $\tau(\supp(x))<\infty$, where $\supp(x)$ denotes the support of $x$. Let $S_{\mathcal{M}}$ be the linear span of $S_{\mathcal{M}}^{+}$. Then every $x\in S_{\mathcal{M}}$ has a finite trace and $S_{\mathcal{M}}$ is a $w^{*}$-dense $*$-subalgebra of $\mathcal{M}$. Let $0<p<\infty$. The noncommutative $L^{p}$-space
associated with $(\mathcal{M}, \tau)$ is denoted as the closure of $S_{\mathcal{M}}$ with respect to the (quasi-)norm
\begin{align*}
 \|x\|_{p}=\left(\tau(|x|^{p})\right)^{\frac{1}{p}},\,\,x\in S_{\mathcal{M}},
\end{align*}
where $|x|=(x^{*}x)^{\frac{1}{2}}$ is the modulus of $x$. By convention, we set $L^{\infty}(\mathcal{M})=\mathcal{M}$ equipped with the operator norm $\|\cdot\|$. The spaces $L^{p}(\mathcal{M})$ can be also described as a subset of measurable operators with respect to $(\mathcal{M}, \tau)$. The reader can find more information on noncommutative-$L^{p}$ spaces in e.g. \cite{px03} and \cite{x07}.

\subsection{Noncommutative Markov semigroups}
\hskip\parindent
Let $M_{n}$ be the algebra of $n$ by $n$ complex matrices and $I_{n}$ is the identity operator on $M_{n}$. An operator $T: \mathcal{M}\rightarrow \mathcal{M}$ is called completely contractive if $T\bigotimes I_{n}: M_{n}(\mathcal{M})\rightarrow M_{n}(\mathcal{M})$ is contractive for every $n\geq1$. An operator $T: \mathcal{M}\rightarrow \mathcal{M}$ is called completely positive if $T\bigotimes I_{n}: M_{n}(\mathcal{M})\rightarrow M_{n}(\mathcal{M})$ is positive for every $n\geq1$.

\smallskip
Let $(\mathcal{M}, \tau)$ be a pair of a von Neumann algebra and a normal finite faithful trace. A family of operators $\mathcal T=(T_{t})_{t\geq0}$ over $\mathcal M$ is said to be a semigroup if $T_{t}T_{s}=T_{t+s}$, $T_{0}=id$. $\mathcal T$ is said to be a noncommutative Markov semigroup if
\begin{itemize}
  \item [(i)] every $T_{t}$ is a normal completely positive map on $\mathcal M$ such that $T_{t}(1)=1$;
  \item [(ii)] every $T_{t}$ is self-adjoint with respect to the trace $\tau$, i.e. $\tau(T_{t}(f^{*})g)=\tau(f^{*}T_{t}(g))$ for any $f$, $g\in\mathcal{M}\cap L^{2}(\mathcal{M})$;
  \item [(iii)] $T_{t}(f)\rightarrow f$ as $t\rightarrow0^{+}$ in the weak-$*$ topology of $\mathcal{M}$.
\end{itemize}

Note that $\mathrm{(i)}$ implies that $T_{t}$ is a complete contraction on $\mathcal M$, in particular,
\begin{align}\label{5151}
  \|T_{t}f\|_{\infty}\leq\|f\|_{\infty}.
\end{align}
$\mathrm{(i)}$ and $\mathrm{(ii)}$ imply that $T_{t}$ is $\tau$-preserving, that is, $\tau(T_{t}f)=\tau(f)$ for all $f\in\mathcal M\cap L^1(\mathcal M)$, so that each $T_{t}$ naturally extends to a contraction on $L^{p}(\mathcal{M})$ for any $1\leq p<\infty$ (see more details in \cite{jx07}).  More information of quantum Markov semigroups can be found in  \cite[Chapter 5]{jmx06}.

A Markov semigroup $\mathcal T=(T_{t})_{t\geq0}$ always admits an infinitesimal generator
\begin{align*}
  A=\lim_{t\rightarrow0}\frac{id-T_{t}}{t} \,\,\,\,\text{with}\,\, \,\,T_{t}=e^{-tA}
\end{align*}
 and its domain in $L^{\infty}(\mathcal{M})$ is defined as
\begin{align*}
 \mathrm{dom}_{\infty}(A)=\left\{f\in \mathcal{M}: \,\,\lim_{t\rightarrow0}\frac{f-T_{t}f}{t}\,\, \text{exists}\right\},
\end{align*}
where the limit is taken in the weak-$*$ topology. It is easy to see that $\frac{1}{s}\int_{0}^{s}T_{t}fdt\in\mathrm{dom}_{\infty}(A)$ for any $s>0$ and $f\in \mathcal{M}$, which implies that $\mathrm{dom}_{\infty}(A)$ is weak-$*$ dense in $\mathcal M$. We write $A_{\infty}$ for the restriction of $A$ on $\mathrm{dom}_{\infty}(A)$. Then one can check that $T_{s}f\in \mathrm{dom}_{\infty}(A)$ for any $s>0$ and $f\in \mathcal{M}$. See e.g. \cite[Page 694]{mj12} for more details.

\begin{definition}\label{p0}
A semigroup $\mathcal T=(T_{t})_{t\geq0}$ is called quasi-decreasing (resp. quasi-increasing) if there exists $\beta\geq 0$ such that $\frac{T_{t}}{t^{\beta}}$ decreases (resp. increases), that is,  $\frac{T_{t}(f)}{t^{\beta}}\geq \frac{T_{s}(f)}{s^{\beta}}$ (resp. $\frac{T_{t}(f)}{t^{\beta}}\leq \frac{T_{s}(f)}{s^{\beta}}$) for any $f\in\mathcal M_+$ and $s>t>0$. We will say a semigroup $(T_{t})_{t\geq0}$ is quasi-monotone if it is either quasi-decreasing or quasi-increasing.
\end{definition}
For a semigroup $\mathcal T=(T_{t})_{t\geq0}$ generated by $A$, the subordinated Poisson semigroup $\mathcal P=(P_{t})_{t \geq 0}$ is defined by $P_{t}=e ^{-t\sqrt{A}}$. It is easy to check that $(P_{t})_{t\geq0}$ is again a Markov semigroup satisfying the same properties (i)-(iii) as $\mathcal T=(T_{t})_{t\geq0}$. Note that $(P_{t})_{t\geq0}$ is chosen such that $(\frac{\partial^{2}}{\partial t^{2}}-A)P_{t}=0.$
 By functional calculus and an elementary identity, we have the following subordination formula (see e.g. \cite[Page 47]{s02}):
\begin{align}\label{2}
  P_{t}=\frac{1}{2\sqrt{\pi}}\int_{0}^{\infty}t e^{-\frac{t^{2}}{4u}}u^{-\frac{3}{2}}T_{u}du.
\end{align}
By Definition \ref{p0}, it is easy to see that the subordinated Poisson semigroup $(P_{t})_{t \geq 0}$ is quasi-decreasing with $\beta=1$. Indeed, since $T_{u}$ is positive and $e^{-\frac{t^{2}}{4u}}u^{-\frac{3}{2}}$ is a function decreasing with respect to $t$, we then have
\begin{align}\label{4}
 P_{t}f\leq\frac{t}{s}P_{s}f
 \end{align}
for any positive $f$ and $0<s\leq t$.

 According to the mean ergodic theorem \cite{ja1,ja2} (see also \rm{\cite[Section 5]{mj12}}), we have the following splitting lemma which is crucial in the present paper.
\begin{lemma}\label{gf}
If the trace $\tau$ is finite, then $(P_{t})_{t\geq0}$ induces a canonical splitting on $\mathcal{M}$ such that
$$\mathcal{M}=\mathcal{M}^{\circ}\oplus ker(A_{\infty}^{\frac{1}{2}}).$$
Here $ker(A_{\infty}^{\frac{1}{2}})=\{f\in \mathrm{dom}_{\infty}(A^{\frac{1}{2}}): A^{\frac{1}{2}}f=0\}=\{f\in \mathcal{M}: \;P_{t}f=f\;\forall t>0\}$ and $\mathcal{M}^{\circ}=\{f\in \mathcal{M}: \,\lim_{t\rightarrow\infty}P_{t}f=0\}$.
In particular, this implies that any $f\in \mathcal{M}$ can be decomposed as
 $$ f=f_0+f_1,\,\,\, f_0\in\mathcal{M}^{\circ},\,\,\,f_1\in ker(A_{\infty}^{\frac{1}{2}}).$$
\end{lemma}
We will need the following Kadison-Schwarz inequalities. For any unital completely positive map $T$ on $\mathcal{M}$,
\begin{align}\label{1}
\left|T(f)\right|^{2} \leq T\left(|f|^{2}\right).
\end{align}
See e.g. \cite[Page 495]{k52}. In particular, for any measure space $(\Omega,d\mu)$,
\begin{align}\label{1.1}
\left| \int_{\Omega}ghd\mu\right|^{2}\leq \int_{\Omega}g^{2}d\mu\int_{\Omega}|h|^{2}d\mu,
\end{align}
where $g:\Omega\rightarrow[0,\infty)$ and $h:\Omega\rightarrow \mathcal{M}$ are functions such that all members in \eqref{1.1} make senses.

The following lemma, as an improvement of \rm{\cite[Corollary 2]{tms19}}, plays a key role in the whole paper.
\begin{proposition}\label{ws}
Let $f\in \mathcal{M}$. For all $j\in\mathbb{N}$ and $t>0$, we have
\begin{align}\label{a1}
  \left|\frac{\partial^{j} P_{t}f}{\partial t^{j}}\right|^{2}\leq \frac{2^{j(j+4)}}{t^{2j}}P_{t}|f|^{2}.
\end{align}
\end{proposition}
\begin{proof}
We first claim that for any $(i,s)\in\mathbb N^*\times\mathbb R_+$, there holds
\begin{align*}
\left|\frac{\partial^{i} P_{s}f}{\partial s^{i}}\right|^{2}
\leq\frac{8}{(\frac{s}{2})^{2}}P_{\frac{s}{2}}\left|\left.\frac{\partial^{i-1} P_uf}{\partial u^{i-1}}\right|_{u=\frac{s}{2}}\right|^{2}.
\end{align*}
Fix $i\geq1$ and $s>0$. By using (\ref{2}) and the Kadison-Schwarz inequality \eqref{1.1}, one gets
\begin{align*}
 \left|\frac{\partial^{i} P_{s}f}{\partial s^{i}}\right|^{2}= &\left|\left.\frac{\partial P_{s_1}}{\partial s_1}\right|_{s_1=\frac{s}{2}}(\left.\frac{\partial^{i-1} P_{s_2}f}{\partial s_2^{i-1}}\right|_{s_2=\frac{s}{2}})\right|^{2}
\\= &\left|\frac{1}{2\sqrt{\pi}}\int_{0}^{\infty}(1-\frac{s^{2}}{8u}) e^{-\frac{s^{2}}{16u}}u^{-\frac{3}{2}}T_{u}(\left.\frac{\partial^{i-1} P_{s_2}f}{\partial s_2^{i-1}}\right|_{s_2=\frac{s}{2}})du\right|^{2}
\\ \leq&2\left|\frac{1}{2\sqrt{\pi}}\int_{0}^{\infty} e^{-\frac{s^{2}}{16u}}u^{-\frac{3}{2}}T_{u}(\left.\frac{\partial^{i-1} P_{s_2}f}{\partial s_2^{i-1}}\right|_{s_2=\frac{s}{2}})du\right|^{2}
\\&+2\left|\frac{1}{2\sqrt{\pi}}\int_{0}^{\infty}\frac{s^{2}}{8u} e^{-\frac{s^{2}}{16u}}u^{-\frac{3}{2}}T_{u}(\left.\frac{\partial^{i-1} P_{s_2}f}{\partial s_2^{i-1}}\right|_{s_2=\frac{s}{2}})du\right|^{2}
\\ =&\frac{2}{(\frac{s}{2})^{2}}\left|\frac{1}{2\sqrt{\pi}}\int_{0}^{\infty} \frac{s}{2}e^{-\frac{s^{2}}{16u}}u^{-\frac{3}{2}}T_{u}(\left.\frac{\partial^{i-1} P_{s_2}f}{\partial s_2^{i-1}}\right|_{s_2=\frac{s}{2}})du\right|^{2}
\\&+2\left|\frac{1}{2\sqrt{\pi}}\int_{0}^{\infty}\frac{s}{4u} (\frac{s}{2}e^{-\frac{s^{2}}{16u}}u^{-\frac{3}{2}})^{\frac{1}{2}}(\frac{s}{2}e^{-\frac{s^{2}}{16u}}u^{-\frac{3}{2}})^{\frac{1}{2}}T_{u}(\left.\frac{\partial^{i-1} P_{s_2}f}{\partial s_2^{i-1}}\right|_{s_2=\frac{s}{2}})du\right|^{2}
\\ \leq&\frac{2}{(\frac{s}{2})^{2}}\left|P_{\frac{s}{2}}(\left.\frac{\partial^{i-1} P_{s_2}f}{\partial s_2^{i-1}}\right|_{s_2=\frac{s}{2}})\right|^{2}
\\&+2\left(\frac{1}{2\sqrt{\pi}}\int_{0}^{\infty}\frac{s^{3}}{32u^{2}}e^{-\frac{s^{2}}{16u}}u^{-\frac{3}{2}}du\right)\left(\frac{1}{2\sqrt{\pi}}\int_{0}^{\infty}\frac{s}{2}e^{-\frac{s^{2}}{16u}}u^{-\frac{3}{2}}\left|T_{u}(\left.\frac{\partial^{i-1} P_{s_2}f}{\partial s_2^{i-1}}\right|_{s_2=\frac{s}{2}})\right|^{2}du\right).
\end{align*}
It is easy to calculate that
$$\left(\frac{1}{2\sqrt{\pi}}\int_{0}^{\infty}\frac{s^{3}}{32u^{2}} e^{-\frac{s^{2}}{16u}}u^{-\frac{3}{2}}du\right)=\frac{16}{s^{2}\sqrt{\pi}}\Gamma(\frac{5}{2})=\frac{3}{(\frac{s}{2})^{2}}.$$
Then by the Kadison-Schwarz inequality \eqref{1} and the subordination formula, we verify the claim
\begin{align*}
\left|\frac{\partial^{i} P_{s}f}{\partial s^{i}}\right|^{2}\leq&\frac{2}{(\frac{s}{2})^{2}}P_{\frac{s}{2}}\left|\left.\frac{\partial^{i-1} P_{s_2}f}{\partial s_2^{i-1}}\right|_{s_2=\frac{s}{2}}\right|^{2}+\frac{6}{(\frac{s}{2})^{2}}\left(\frac{1}{2\sqrt{\pi}}\int_{0}^{\infty}\frac{s}{2}e^{-\frac{s^{2}}{16u}}u^{-\frac{3}{2}}T_{u}\left|(\left.\frac{\partial^{j-1} P_{s_2}f}{\partial s_2^{i-1}}\right|_{s_2=\frac{s}{2}})\right|^{2}du\right)
\\
=&\frac{2}{(\frac{s}{2})^{2}}P_{\frac{s}{2}}\left|\left.\frac{\partial^{j-1} P_{s_2}f}{\partial s_2^{i-1}}\right|_{s_2=\frac{s}{2}}\right|^{2}+\frac{6}{(\frac{s}{2})^{2}}P_{\frac{s}{2}}\left|\left.\frac{\partial^{i-1} P_{s_2}f}{\partial s_2^{i-1}}\right|_{s_2=\frac{s}{2}}\right|^{2}=\frac{8}{(\frac{s}{2})^{2}}P_{\frac{s}{2}}\left|\left.\frac{\partial^{i-1} P_{s_2}f}{\partial s_2^{i-1}}\right|_{s_2=\frac{s}{2}}\right|^{2}.
\end{align*}
Now for $j\in\mathbb N$ and $t>0$,
applying repeatedly the claim to $(i,s)=(j,t),\ldots, (1, \frac{t}{2^{j-1}})$, and finally using the Kadison-Schwarz inequality \eqref{1} again, we get
\begin{align*}
\left|\frac{\partial^{j} P_{t}f}{\partial t^{j}}\right|^{2} \leq&\frac{8}{(\frac{t}{2})^{2}}\cdots \frac{8}{(\frac{t}{2^{j-1}})^{2}} \frac{8}{(\frac{t}{2^{j}})^{2}}P_{\frac{t}{2}+\cdots\frac{t}{2^{j-1}}+\frac{t}{2^{j}}}\left| P_{\frac{t}{2^{j}}}f\right|^{2}
\\ \leq&\frac{8}{(\frac{t}{2})^{2}}\cdots \frac{8}{(\frac{t}{2^{j-1}})^{2}} \frac{8}{(\frac{t}{2^{j}})^{2}} P_{\frac{t}{2}+\cdots\frac{t}{2^{j-1}}+\frac{t}{2^{j}}+\frac{t}{2^{j}}}\left| f\right|^{2}
\\=&\frac{8^{j}2^{j(j+1)}}{t^{2j}}P_{t}|f|^{2}=\frac{2^{j(j+4)}}{t^{2j}}P_{t}|f|^{2}.
\end{align*}
The desired assertion is proved.
\end{proof}

\section{Lipschitz spaces via semigroups}
\hskip\parindent
 Let $(P_{t})_{t\geq0}$ be the subordinated Poisson semigroup of a Markov semigroup $(T_{t})_{t\geq0}$. We introduce the Lipschitz spaces associated with $(P_{t})_{t\geq0}$ and present some properties that will be useful in the study of the Campanato spaces.

  For all $\alpha>0$, we define
\begin{align*}
  \Lambda_{\alpha}(\mathcal{P})=\left\{f\in\mathcal{M}:\,\,\,\|f\|_{\Lambda_{\alpha}(\mathcal{P})}<\infty\right\},
\end{align*}
with
\begin{align}\label{e33}
 \|f\|_{\Lambda_{\alpha}(\mathcal{P})}=\|f\|_{\infty}+\|f\|_{\dot{\Lambda}_{\alpha}(\mathcal{P})},\quad\|f\|_{\dot{\Lambda}_{\alpha}(\mathcal{P})}=\sup_{t>0}\frac{1}{t^{\alpha-([\alpha]+1)}}\left\|\frac{\partial^{[\alpha]+1}P_{t}f}{\partial t^{[\alpha]+1}}\right\|_{\infty}.
\end{align}
One can easily check that $\left\|f^{*}\right\|_{\Lambda_{\alpha}({\mathcal P)}}=\left\|f\right\|_{\Lambda_{\alpha}({\mathcal P})}$ from the positivity of the Poisson semigroup.

\begin{proposition}\label{lael}
Let $\alpha>0$. Then
\begin{itemize}
  \item [\rm{(i)}] $\left\|\cdot\right\|_{\Lambda_{\alpha}({\mathcal P})}$ is a norm on $\mathcal{M}$;
  \item [\rm{(ii)}] the space $\Lambda_{\alpha}(\mathcal{P})$ is complete with respect to the norm $\left\|\cdot\right\|_{\Lambda_{\alpha}({\mathcal P})}$.
\end{itemize}
\end{proposition}
\begin{proof}
\rm{(i)} is obvious.

\rm{(ii)} Let $\{f_{n}\}_{n\geq1}$ be a Cauchy sequence in $\Lambda_{\alpha}(\mathcal{P})$. By the definition, $\{f_{n}\}_{n\geq1}$ is automatically a Cauchy sequence in $\mathcal{M}$. So there exists a unique $f\in\mathcal{M}$ such that $f_{n}\rightarrow f$ in $\mathcal{M}$. This implies that, for $t>0$ fixed, we can find an integer $m$ sufficiently large such that $\|f-f_{m}\|_{\infty}<t^{\alpha}$. By Proposition \ref{ws} and (\ref{5151}), we get
\begin{align*}
  \frac{1}{t^{\alpha-([\alpha]+1)}}\left\|\frac{\partial^{[\alpha]+1}P_{t}f}{\partial t^{[\alpha]+1}}\right\|_{\infty}
\leq&\frac{1}{t^{\alpha-([\alpha]+1)}}\left\|\frac{\partial^{[\alpha]+1}P_{t}}{\partial t^{[\alpha]+1}}(f-f_{m})\right\|_{\infty}+\frac{1}{t^{\alpha-([\alpha]+1)}}\left\|\frac{\partial^{[\alpha]+1}P_{t}f_{m}}{\partial t^{[\alpha]+1}}\right\|_{\infty}
\\=&\frac{1}{t^{\alpha-([\alpha]+1)}}\left\|\left|\frac{\partial^{[\alpha]+1}P_{t}}{\partial t^{[\alpha]+1}}(f-f_{m})\right|^{2}\right\|_{\infty}^{\frac{1}{2}}+\frac{1}{t^{\alpha-([\alpha]+1)}}\left\|\frac{\partial^{[\alpha]+1}P_{t}f_{m}}{\partial t^{[\alpha]+1}}\right\|_{\infty}
\\ \lesssim&_{\alpha}\frac{1}{t^{\alpha-([\alpha]+1)}}\left\|\frac{P_{t}}{ t^{2([\alpha]+1)}}\left|f-f_{m}\right|^{2}\right\|_{\infty}^{\frac{1}{2}}+\|f_{m}\|_{\Lambda_{\alpha}(\mathcal{P})}
\\ =&\frac{1}{t^{\alpha}}\left\|P_{t} \left|f-f_{m}\right|^{2}\right\|_{\infty}^{\frac{1}{2}}+\|f_{m}\|_{\Lambda_{\alpha}(\mathcal{P})}
\\ \leq&\frac{1}{t^{\alpha}}\left\| f-f_{m}\right\|_{\infty}+\|f_{m}\|_{\Lambda_{\alpha}(\mathcal{P})}\leq1+\sup_{n\geq1}\|f_{n}\|_{\Lambda_{\alpha}(\mathcal{P})}<\infty.
\end{align*}
In the last inequality, we have used the fact that the norms of $f_n$'s are uniformly bounded which is a trivial consequence of a Cauchy sequence.
This yields $f\in \Lambda_{\alpha}(\mathcal{P})$.

Now we show that the sequence $\{f_{n}\}_{n\geq1}$ converges to $f$ in $\Lambda_{\alpha}(\mathcal{P})$. Let $\epsilon, t>0$. Note that $\{f_{n}\}_{n\geq1}$ is a Cauchy sequence in $\Lambda_{\alpha}(\mathcal{P})$, so there exists a large integer $N_{1}$ such that $\|f_{n_1}-f_{n_2}\|_{\Lambda_{\alpha}(\mathcal{P})}<\epsilon$ for every $n_1,n_2\geq N_{1}$. On the other hand, since $f_{n}\rightarrow f$ in $\mathcal{M}$, we can not only take an integer $N_{2}$ sufficiently large such that $\|f-f_{n}\|_{\infty}<\epsilon$ for every $n\geq N_{2}$, but also another integer $m$ big enough, for instance $m\geq\max\{N_{1},N_{2}\}$, such that $\|f-f_{m}\|_{\infty}< t^{\alpha}\epsilon$.
Therefore, together with Proposition \ref{ws}, for $n\geq\max\{N_{1},N_{2}\}$, we have
\begin{align*}
 \frac{1}{t^{\alpha-([\alpha]+1)}}\left\|\frac{\partial^{[\alpha]+1}P_{t}}{\partial t^{[\alpha]+1}}(f-f_{n})\right\|_{\infty}
\leq&\frac{1}{t^{\alpha-([\alpha]+1)}}\left\|\frac{\partial^{[\alpha]+1}P_{t}}{\partial t^{[\alpha]+1}}(f-f_{m})\right\|_{\infty}
\\&+\frac{1}{t^{\alpha-([\alpha]+1)}}\left\|\frac{\partial^{[\alpha]+1}P_{t}}{\partial t^{[\alpha]+1}}(f_{m}-f_{n})\right\|_{\infty}
\\ \lesssim&_{\alpha}\frac{1}{t^{\alpha}}\left\| f-f_{m}\right\|_{\infty}+\|f_{n}-f_{m}\|_{\Lambda_{\alpha}(\mathcal{P})}\leq2\epsilon,
\end{align*}
which implies that $\|f_n-f\|_{\Lambda_{\alpha}(\mathcal{P})}\lesssim\epsilon$ by the arbitrariness of $t>0$.
\end{proof}
The following characterization of the seminorm $\|\cdot\|_{\dot{\Lambda}_{\alpha}(\mathcal{P})}$ will play an important role in the present paper.

\begin{lemma}\label{190}
Let $(P_{t})_{t\geq0}$ be the Poisson semigroup subordinated to a Markov semigroup $(T_{t})_{t\geq0}$. For any $f\in \mathcal{M}$ and $\alpha>0$, we have
\begin{itemize}
  \item [\rm({i})] $\sup_{t>0}\frac{1}{t^{\alpha-([ \alpha]+1)}}\left\|\frac{\partial^{[ \alpha]+1}P_{t}}{\partial t^{[\alpha]+1}}(f-P_{t}f)\right\|_{\infty}\leq2\|f\|_{\dot{\Lambda}_{\alpha}(\mathcal{P})};$
  \item [\rm({ii})]
 $\|f\|_{\dot{\Lambda}_{\alpha}(\mathcal{P})}\leq\frac{1}{1-2^{\alpha-([ \alpha]+1)}}\sup_{t>0}\frac{1}{t^{\alpha-([ \alpha]+1)}}\left\|\frac{\partial^{[ \alpha]+1}P_{t}}{\partial t^{[\alpha]+1}}(f-P_{t}f)\right\|_{\infty}.$
\end{itemize}
\end{lemma}
\begin{proof}
To prove \rm({i}), we note that
\begin{align*}
  \frac{\partial P_{s}}{\partial s}P_{t}=\left.\frac{\partial P_{u}}{\partial u}\right|_{u=t+s}
\end{align*}
holds for any $s,t>0$. Then for a fixed $t>0$, by the triangle inequality, we have
\begin{align*}
  &\frac{1}{t^{\alpha-([ \alpha]+1)}}\left\|\frac{\partial^{[ \alpha]+1}P_{t}}{\partial t^{[\alpha]+1}}(f-P_{t}f)\right\|_{\infty}
  \\ \leq&\frac{1}{t^{\alpha-([ \alpha]+1)}}\left\|\frac{\partial^{[ \alpha]+1}P_{t}f}{\partial t^{[\alpha]+1}}\right\|_{\infty}
  +\frac{1}{t^{\alpha-([ \alpha]+1)}}\left\|\frac{\partial^{[ \alpha]+1}P_{t}}{\partial t^{[\alpha]+1}}P_{t}f\right\|_{\infty}
  \\ =&\frac{1}{t^{\alpha-([ \alpha]+1)}}\left\|\frac{\partial^{[ \alpha]+1}P_{t}f}{\partial t^{[\alpha]+1}}\right\|_{\infty}
 +2^{\alpha-([ \alpha]+1)}\frac{1}{(2t)^{\alpha-([ \alpha]+1)}}\left\|\left.\frac{\partial^{[ \alpha]+1}P_{s}f}{\partial s^{[\alpha]+1}}\right|_{s=2t}\right\|_{\infty}\\
 \leq& \|f\|_{\dot{\Lambda}_{\alpha}(\mathcal{P})}+2^{\alpha-([ \alpha]+1)}\|f\|_{\dot{\Lambda}_{\alpha}(\mathcal{P})}.
\end{align*}
Taking the supremum over $t>0$ on the left-hand side, we get \rm(i).

To prove \rm({ii}), by the triangle inequality, we obtain for a fixed $t>0$,
\begin{align*}
  \frac{1}{t^{\alpha-([ \alpha]+1)}}\left\|\frac{\partial^{[ \alpha]+1}P_{t}f}{\partial t^{[\alpha]+1}}\right\|_{\infty}
 \leq&\frac{1}{t^{\alpha-([ \alpha]+1)}}\left\|\frac{\partial^{[ \alpha]+1}P_{t}}{\partial t^{[\alpha]+1}}(f-P_{t}f)\right\|_{\infty}
  +\frac{1}{t^{\alpha-([ \alpha]+1)}}\left\|\frac{\partial^{[ \alpha]+1}P_{t}}{\partial t^{[\alpha]+1}}P_{t}f\right\|_{\infty}
\\ \leq&\sup_{t>0}\frac{1}{t^{\alpha-([ \alpha]+1)}}\left\|\frac{\partial^{[ \alpha]+1}P_{t}}{\partial t^{[\alpha]+1}}(f-P_{t}f)\right\|_{\infty}+2^{\alpha-([ \alpha]+1)}\|f\|_{\dot{\Lambda}_{\alpha}(\mathcal{P})}.
\end{align*}
Taking the supremum over $t>0$ on the left side and then using the condition that $2^{\alpha-([ \alpha]+1)}<1$ trivially holds for $\alpha>0$, we deduce
$$
  \|f\|_{\dot{\Lambda}_{\alpha}(\mathcal{P})}\leq\frac{1}{1-2^{\alpha-([ \alpha]+1)}}\sup_{t>0}\frac{1}{t^{\alpha-([ \alpha]+1)}}\left\|\frac{\partial^{[ \alpha]+1}P_{t}}{\partial t^{[\alpha]+1}}(f-P_{t}f)\right\|_{\infty}.
$$
\end{proof}
As a matter of fact, just like the classical case (see e.g. \cite[Chapter V,\,Lemma 5]{s70}), the index $[\alpha]+1$ in (\ref{e33}) can be replaced by any integer greater than $\alpha$.
Let $k$ be an integer such that $k>\alpha>0$, we define $\Lambda_{\alpha,k}({\mathcal P})$ as
  $$\Lambda_{\alpha,k}({\mathcal P})=\left\{f\in\mathcal{M}:\,\,  \left\|f\right\|_{\Lambda_{\alpha,k}({\mathcal P})}<\infty\right\},$$
with
 $$\left\|f\right\|_{\Lambda_{\alpha,k}({\mathcal P})}=\|f\|_{\infty}+\left\|f\right\|_{\dot{\Lambda}_{\alpha,k}({\mathcal P})},\quad \left\|f\right\|_{\dot{\Lambda}_{\alpha,k}({\mathcal P})}=\sup_{t>0}\frac{1}{t^{ \alpha-k}}\left\|\frac{\partial^{k}P_{t}f}{\partial t^{k}}\right\|_{\infty}.$$
Similarly, it is not difficult to get that $\|\cdot\|_{\Lambda_{\alpha,k}({\mathcal P})}$ is a norm and the space $\Lambda_{\alpha,k}(\mathcal{P})$ is complete with respect to the norm $\left\|\cdot\right\|_{\Lambda_{\alpha,k}({\mathcal P})}$.

The following proposition tells that, for $\alpha>0$ fixed, the spaces $\Lambda_{\alpha,k}({\mathcal P})$ are all the same when $k>\alpha$.
\begin{proposition}\label{20}
Let $\alpha>0$ and $k$ be an integer such that $k\geq[\alpha]+2$. Then, for $f\in \mathcal{M}$, we have
  \begin{align}\label{h1}
  \left\|f\right\|_{\Lambda_{\alpha}({\mathcal P})}\leq\max\{1,\frac{1}{(([\alpha]+2)-(\alpha+1))\cdots((k-1)-(\alpha+1))(k-(\alpha+1))}\}\left\|f\right\|_{\Lambda_{\alpha,k}({\mathcal P})}
 \end{align}
 and
 \begin{align}\label{h2}
  \left\|f\right\|_{\Lambda_{\alpha,k}({\mathcal P})}\lesssim2^{(k-[\alpha]-1)(\frac{k+[\alpha]+2}{2}-\alpha)}\left\|f\right\|_{\Lambda_{\alpha}({\mathcal P})}.
 \end{align}
\end{proposition}
\begin{proof}
We first deal with \eqref{h1} and divide the proof into two steps.

\textbf{Step 1.} We will show that it suffices to estimate the following inequality on $\mathcal M^{\circ}$: for $g\in\mathcal M^{\circ}$,
\begin{align}\label{h3}
  \left\|g\right\|_{\dot{\Lambda}_{\alpha}({\mathcal P})}\leq\frac{1}{(([\alpha]+2)-(\alpha+1))\cdots((k-1)-(\alpha+1))(k-(\alpha+1))}\left\|g\right\|_{\dot{\Lambda}_{\alpha,k}({\mathcal P})}.
\end{align}
Let $f\in \mathcal M$. By Lemma \ref{gf}, we can split $f$ as
\begin{align*}
  f=f_{0}+f_{1},\;\;\;f_{0}\in \mathcal M^{\circ},\;f_{1}\in ker(A^{{\frac{1}{2}}}_{\infty}).
\end{align*}
By definition,  $f_{1}\in ker(A^{{\frac{1}{2}}}_{\infty})$ implies $P_{t}f_{1}=f_{1}$ for any $t>0$. Thus, we have
\begin{align}\label{h4}
\left\|f\right\|_{\dot{\Lambda}_{\alpha}({\mathcal P})}&=\sup_{t>0}\frac{1}{t^{\alpha-([\alpha]+1)}}\left\|\frac{\partial^{[\alpha]+1}P_{t}}{\partial t^{[\alpha]+1}}(f_{0}+f_{1})\right\|_{\infty}\\
\nonumber& =\sup_{t>0}\frac{1}{t^{\alpha-([\alpha]+1)}}\left\|\frac{\partial^{[\alpha]+1}P_{t}f_{0}}{\partial t^{[\alpha]+1}}\right\|_{\infty}=\left\|f_{0}\right\|_{\dot{\Lambda}_{\alpha}({\mathcal P})}.
\end{align}
In a similar way, we  obtain
\begin{align}\label{h5}
  \left\|f\right\|_{\dot{\Lambda}_{\alpha,k}({\mathcal P})}=\left\|f_{0}\right\|_{\dot{\Lambda}_{\alpha,k}({\mathcal P})}.
\end{align}
Thus, applying $\eqref{h3}$ to $f_0\in\mathcal M^{\circ}$, we may get the desired \eqref{h1},
\begin{align*}
\left\|f\right\|_{\Lambda_{\alpha}({\mathcal P})}=& \|f\|_{\infty}+\left\|f\right\|_{\dot{\Lambda}_{\alpha}({\mathcal P})}
=\|f\|_{\infty}+\left\|f_{0}\right\|_{\dot{\Lambda}_{\alpha}({\mathcal P})}
  \\ \leq&\|f\|_{\infty}+\frac{1}{(([\alpha]+2)-(\alpha+1))\cdots((k-1)-(\alpha+1))(k-(\alpha+1))}\left\|f_{0}\right\|_{\dot{\Lambda}_{\alpha,k}({\mathcal P})}
\\ =&\|f\|_{\infty}+\frac{1}{(([\alpha]+2)-(\alpha+1))\cdots((k-1)-(\alpha+1))(k-(\alpha+1))}\left\|f\right\|_{\dot{\Lambda}_{\alpha,k}({\mathcal P})}
\\ \leq&\max\{1,\frac{1}{(([\alpha]+2)-(\alpha+1))\cdots((k-1)-(\alpha+1))(k-(\alpha+1))}\}\left\|f\right\|_{\Lambda_{\alpha,k}({\mathcal P})}.
\end{align*}

\textbf{Step 2.} We now show (\ref{h3}). Let $f\in\mathcal M^\circ$. By the definition of $\mathcal M^\circ$, for a fixed $t>0$, one has
\begin{align*}
 -P_{t}f=\int^{\infty}_{t}\frac{\partial P_{s}f}{\partial s}ds.
\end{align*}
Let $\ell\geq [\alpha]+2$.
Using the identity
$\frac{\partial^{\ell-1}P_{t}f}{\partial t^{\ell-1}}=P_{\frac{t}{2}}\left.\frac{\partial^{\ell-1}P_{t_{1}}f}{\partial t_{1}^{\ell-1}}\right|_{t_{1}=\frac{t}{2}},$ we may compute
\begin{align*}
\left\|\frac{\partial^{\ell-1}P_{t}f}{\partial t^{\ell-1}}\right\|_{\infty}
=\left\|P_{\frac{t}{2}}\left.\frac{\partial^{\ell-1}P_{t_{1}}f}{\partial t_{1}^{\ell-1}}\right|_{t_{1}=\frac{t}{2}}\right\|_{\infty}
=\left\|\left.\frac{\partial^{\ell-1}P_{t_{1}}}{\partial t_{1}^{\ell-1}}\right|_{t_{1}=\frac{t}{2}}\int^{\infty}_{\frac{t}{2}}\frac{\partial P_{u}f}{\partial u}du\right\|_{\infty}.
\end{align*}
 By the elementary identity $\left.\frac{\partial^{\ell-1}P_{t_{1}}}{\partial t_{1}^{\ell-1}}\right|_{t_{1}=\frac{t}{2}}\frac{\partial P_{u}f}{\partial u}=\frac{\partial^{\ell} P_{u+\frac{t}{2}}f}{\partial u^{\ell}}$ and taking $s=u+\frac{t}{2}$, we have
\begin{align*}
\frac{1}{t^{\alpha-({\ell}-1)}}\left\|\frac{\partial^{\ell-1}P_{t}f}{\partial t^{\ell-1}}\right\|_{\infty}
=&\frac{1}{t^{\alpha-(\ell-1)}}\left\|\int^{\infty}_{\frac{t}{2}}\frac{\partial^{\ell} P_{u+\frac{t}{2}}f}{\partial u^{\ell}}du\right\|_{\infty}
\\=&\frac{1}{t^{\alpha-(\ell-1)}}\left\|\int^{\infty}_{t}\frac{\partial^{\ell}P_{s}f}{\partial s^{\ell}}ds\right\|_{\infty}
\\ \leq&\frac{1}{t^{\alpha-(\ell-1)}}\int^{\infty}_{t}\left\|\frac{\partial^{\ell}P_{s}f}{\partial s^{\ell}}\right\|_{\infty}ds
\\=&\frac{1}{t^{\alpha-(\ell-1)}}\int^{\infty}_{t}\frac{s^{\alpha-\ell}}{s^{\alpha-\ell}}\left\|\frac{\partial^{\ell}P_{s}f}{\partial s^{\ell}}\right\|_{\infty}ds
\\ \leq&\sup_{s>0}\left\{\frac{1}{s^{\alpha-\ell}}\left\|\frac{\partial^{l}P_{s}f}{\partial s^{\ell}}\right\|_{\infty}\right\}\frac{1}{t^{\alpha-(\ell-1)}}\int^{\infty}_{t}s^{\alpha-\ell}ds
\\=&\frac{1}{\ell-(\alpha+1)}\left\|f\right\|_{\dot{\Lambda}_{\alpha,\ell}({\mathcal P})}.
\end{align*}
Taking supremum over $t>0$ on the left-hand side, we get
\begin{align}\label{for iteration0}
  \left\|f\right\|_{\dot{\Lambda}_{\alpha,\ell-1}({\mathcal P})}\leq\frac{1}{\ell-(\alpha+1)}\left\|f\right\|_{\dot{\Lambda}_{\alpha,\ell}({\mathcal P})}.
\end{align}
By iteration, that is, applying \eqref{for iteration0} to $\ell=[\alpha]+2,\dotsm,k$,
we deduce the desired estimate (\ref{h3}).

\bigskip

We now prove \eqref{h2}. Let $\ell\geq[\alpha]+2$. For one $t>0$ fixed, we use Proposition \ref{ws} and obtain
\begin{align}\label{21}
\frac{1}{t^{\alpha-\ell}}\left\|\frac{\partial^{\ell}P_{t}f}{\partial t^{\ell}}\right\|_{\infty}
=&\frac{1}{t^{\alpha-\ell}}\left\|\left|\left.\frac{\partial P_{v}}{\partial v}\right|_{v=\frac{t}{2}}\left.\frac{\partial^{\ell-1}P_{s}f}{\partial s^{\ell-1}}\right|_{s=\frac{t}{2}}\right|^{2}\right\|^{\frac{1}{2}}_{\infty}
\\ \nonumber \lesssim&\frac{2}{t^{\alpha-(\ell-1)}}\left\|P_{\frac{t}{2}}\left|\left.\frac{\partial^{\ell-1}P_{s}f}{\partial s^{\ell-1}}\right|_{s=\frac{t}{2}}\right|^{2}\right\|^{\frac{1}{2}}_{\infty}
\\ \nonumber \leq&2^{\ell-\alpha}\sup_{t>0}\frac{1}{t^{\alpha-(\ell-1)}}\left\|P_{t}\left|\frac{\partial^{\ell-1}P_{t}f}{\partial t^{\ell-1}}\right|^{2}\right\|^{\frac{1}{2}}_{\infty}
\\ \nonumber \leq&2^{\ell-\alpha}\sup_{t>0}\frac{1}{t^{\alpha-(\ell-1)}}\left\|\frac{\partial^{\ell-1}P_{t}f}{\partial t^{\ell-1}}\right\|_{\infty}.
\end{align}
Taking the supremum over $t>0$ on the left-hand side, we have
\begin{align*}
 \left\|f\right\|_{\dot{\Lambda}_{\alpha,\ell}({\mathcal P})}\lesssim&2^{\ell-\alpha} \left\|f\right\|_{\dot{\Lambda}_{\alpha,\ell-1}({\mathcal P})},
\end{align*}
which implies
\begin{align}\label{dnl}
 \left\|f\right\|_{\Lambda_{\alpha,\ell}({\mathcal P})}=&\|f\|_{\infty}+\left\|f\right\|_{\dot{\Lambda}_{\alpha,\ell}({\mathcal P})}
\lesssim\|f\|_{\infty}+2^{\ell-\alpha}\left\|f\right\|_{\dot{\Lambda}_{\alpha,\ell-1}({\mathcal P})}
 \leq2^{\ell-\alpha}\left\|f\right\|_{\Lambda_{\alpha,\ell-1}({\mathcal P})}.
\end{align}
Applying (\ref{dnl}) to $\ell=k,\dotsm,[\alpha]+2$, we obtain (\ref{h2}) by iteration.

\end{proof}

\section{Campanato Spaces: Proof of Theorem \ref{e25}}
\hskip\parindent
In this section, we introduce the Campanato spaces associated with the quantum Poisson semigroup $(P_{t})_{t\geq0}$ and show the surprising coincidence of the column Campanato space and the row Campanato space.

Let $(T_{t})_{t\geq0}$ be a Markov semigroup acting on a finite von Neumann algebra $\mathcal{M}$ and $(P_{t})_{t\geq0}$ be the subordinated Poisson semigroup. For all $\alpha>0$, the column Campanato space is given by
 \begin{align*}
   \mathcal{L}_{\alpha}^{c}(\mathcal P)=\left\{f\in\mathcal{M}: \,\, \left\|f\right\|_{\mathcal{L}_{\alpha}^{c}(\mathcal P)}<\infty\right\},
 \end{align*}
where
\begin{align}\label{e34}
    \left\|f\right\|_{\mathcal L_{\alpha}^{c}(\mathcal P)}=\|f\|_{\infty}+\left\|f\right\|_{\dot{\mathcal L}_{\alpha}^{c}(\mathcal P)},\quad \left\|f\right\|_{\dot{\mathcal L}_{\alpha}^{c}(\mathcal P)}=\sup_{t>0}\frac{1}{t^{\alpha}}
   \left\|P_{t}|(I-P_{t})^{[\alpha]+1}f|^{2}\right\|_{\infty}^{\frac{1}{2}}.
\end{align}
Argued as in the case of $\alpha=0$ (see e.g. \rm{\cite[Proposition 2.1]{mj12}}), we know that $\left\|\cdot\right\|_{\dot{\mathcal L}_{\alpha}^{c}(\mathcal P)}$ is a semi-norm, hence $\left\|\cdot\right\|_{\mathcal L_{\alpha}^{c}(\mathcal P)}$ is a norm. Similarly, one defines the row space $  \mathcal{L}_{\alpha}^{r}(\mathcal P)$ with norm $\left\|f\right\|_{\mathcal L_{\alpha}^{r}(\mathcal P)}=\left\|f^{*}\right\|_{\mathcal L_{\alpha}^{c}(\mathcal P)}$. Then the mixture space $  \mathcal{L}_{\alpha}^{cr}(\mathcal P)$ is defined as the intersection $  \mathcal{L}_{\alpha}^{c}(\mathcal P)\cap \mathcal{L}_{\alpha}^{r}(\mathcal P)$.

\begin{proposition}\label{5111}
Suppose $\alpha>0$. Then,
the space $\mathcal L_{\alpha}^{\dag}(\mathcal P)$ is complete with respect to the norm $\|\cdot\|_{\mathcal L_{\alpha}^{\dag}(\mathcal P)}$, where $\dag=\{c,r,cr\}$.
\end{proposition}
\begin{proof}
We only consider the situation $\dag=c$ and the other two cases are handled similarly.

It is easy to see that for any $t>0$ and $g\in\mathcal M$
\begin{align}\label{1615}
  \left\|P_{t}|(I-P_{t})^{[\alpha]+1}g|^{2}\right\|_{\infty}^{\frac{1}{2}}\lesssim_{\alpha}\|g\|_{\infty}.
\end{align}
Indeed, by (\ref{5151}) and the triangle inequality, (\ref{1615}) can be easily iterated as follows:
\begin{align*}
 \left\|P_{t}|(I-P_{t})^{[\alpha]+1}g|^{2}\right\|_{\infty}^{\frac{1}{2}}
 \leq\left\|(I-P_{t})^{[\alpha]+1}g\right\|_{\infty}
\leq 2\left\|(I-P_{t})^{[\alpha]}g\right\|_{\infty}
 \leq2^{[\alpha]+1}\left\|g\right\|_{\infty}.
\end{align*}

Let $\{f_{n}\}_{n\geq1}$ be a Cauchy sequence in $\mathcal L_{\alpha}^{c}(\mathcal P)$. By the definition of $\mathcal L_{\alpha}^{c}(\mathcal P)$, we obtain that $\{f_{n}\}_{n\geq1}$ is also a Cauchy sequence in $\mathcal M$. So there exists a unique $f$ such that $f_n\rightarrow f$ in the operator norm. Thus, for $t>0$ fixed, we can take an integer $m$ sufficiently large, so that $\|f-f_{m}\|_{\infty}<t^{\alpha}$. By (\ref{1615}) and Proposition \ref{ws}, we have
\begin{align*}
  \frac{1}{t^{\alpha}}
   \left\|P_{t}|(I-P_{t})^{[\alpha]+1}f|^{2}\right\|_{\infty}^{\frac{1}{2}}
  \leq&\frac{1}{t^{\alpha}}
   \left\|P_{t}|(I-P_{t})^{[\alpha]+1}(f-f_{m})|^{2}\right\|_{\infty}^{\frac{1}{2}}+\frac{1}{t^{\alpha}}
   \left\|P_{t}|(I-P_{t})^{[\alpha]+1}f_{m}|^{2}\right\|_{\infty}^{\frac{1}{2}}
   \\ \lesssim&_{\alpha}\frac{1}{t^{\alpha}}
   \|f-f_{m}\|_{\infty}+
  \left\|f_{m}\right\|_{\mathcal L_{\alpha}^{c}(\mathcal P)}\leq1+\sup_{n\geq1}\left\|f_{n}\right\|_{\mathcal L_{\alpha}^{c}(\mathcal P)}<\infty.
\end{align*}
The last inequality is true because the norms of $f_{n}$'s are uniformly bounded by the fact that $\{f_{n}\}_{n\geq1}$ is a Cauchy sequence in $\mathcal L_{\alpha}^{c}(\mathcal P)$. This yields the assertion $f\in\mathcal L_{\alpha}^{c}(\mathcal P)$.

Now, we show that the convergence happens in the $\mathcal L_{\alpha}^{c}(\mathcal P)$-norm. Let $\epsilon,t>0$ fixed. By the fact that $\{f_{n}\}_{n\geq1}$ is a Cauchy sequence in $\mathcal L_{\alpha}^{c}(\mathcal P)$, there exists a large integer $K_{1}$ such that
\begin{align}\label{1517}
  \|f_{n_{1}}-f_{n_{2}}\|_{\mathcal L_{\alpha}^{c}(\mathcal P)}<\epsilon
\end{align}
for every $n_{1},n_{2}\geq K_{1}$. Since the sequence $\{f_{n}\}_{n\geq1}$ converges to $f$ in $\mathcal M$, we can take a large integer $K_{2}$ such that $\|f-f_{n}\|_{\infty}<\epsilon$ for every $n\geq K_{2}$ and find an integer $m$ satisfying $m\geq\max\{K_{1},K_{2}\}$ such that $\|f-f_{m}\|_{\infty}<t^{\alpha}\epsilon$. Applying (\ref{1517}) and Proposition \ref{ws}, for $n\geq\max\{K_{1},K_{2}\}$, we have
\begin{align*}
 \frac{1}{t^{\alpha}}
   \left\|P_{t}|(I-P_{t})^{[\alpha]+1}(f-f_{n})|^{2}\right\|_{\infty}^{\frac{1}{2}}
\leq&\frac{1}{t^{\alpha}}
   \left\|P_{t}|(I-P_{t})^{[\alpha]+1}(f-f_{m})|^{2}\right\|_{\infty}^{\frac{1}{2}}
   \\&+\frac{1}{t^{\alpha}}
   \left\|P_{t}|(I-P_{t})^{[\alpha]+1}(f_{m}-f_{n})|^{2}\right\|_{\infty}^{\frac{1}{2}}
\\ \lesssim&_{\alpha}
\frac{1}{t^{\alpha}}
   \|f-f_{n}\|_{\infty}+\|f_{n}-f_{m}\|_{\mathcal L_{\alpha}^{c}(\mathcal P)}<2\epsilon.
\end{align*}
Therefore, we obtain that $\|f-f_{n}\|_{\mathcal L_{\alpha}^{c}(\mathcal P)}\lesssim\epsilon$ since $t>0$ is arbitrary. This complete this proof.
\end{proof}

By convention, we set $C_{M}^{N}=\frac{M!}{(M-N)!N!}$ for any $0\leq N\leq M$. We will need the following lemma in this paper.

\begin{lemma}\label{531}
For any $M\geq1$, we have
\begin{align*}
  (I-P_{t})^{M}=(I-P_{\frac{t}{2}})^{M}(I+\sum_{N=1}^{M}C_{M}^{N}P_{\frac{Nt}{2}}).
\end{align*}
\end{lemma}
\begin{proof} The proof is easy. By the fact that $(P_{\frac{t}{2}})^{N}=P_{\frac{Nt}{2}}$, we may compute
  \begin{align*}
 (I-P_{t})^{M}=&(I-P_{\frac{t}{2}}+P_{\frac{t}{2}}-P_{t})^{M}=(I-P_{\frac{t}{2}})^{M}(I+P_{\frac{t}{2}})^{M}
 \\=&(I-P_{\frac{t}{2}})^{M}(I+\sum_{N=1}^{M}C_{M}^{N}(P_{\frac{t}{2}})^{N})=(I-P_{\frac{t}{2}})^{M}(I+\sum_{N=1}^{M}C_{M}^{N}P_{\frac{Nt}{2}}).
\end{align*}
\end{proof}

We are now ready to prove Theorem \ref{e25}.
\begin{proof}[Proof of Theorem \ref{e25}]
We first deal with \eqref{b2}: for any $\alpha>0$,
$$ \left\|f\right\|_{\mathcal L_{\alpha}^{c}(\mathcal P)} \lesssim_\alpha\left\|f\right\|_{\Lambda_{\alpha}(\mathcal P)}.$$

We study first the situation $0<\alpha<1$. Let $t>0$.
 Since $P_{t}$ is contractive on $\mathcal{M}$, we deduce
 \begin{align*}
  \frac{1}{t^{\alpha}} \left\|P_{t}|f-P_{t}f|^{2}\right\|_{\infty}^{\frac{1}{2}}
 \leq&\frac{1}{t^{\alpha}} \left\|f-P_{t}f\right\|_{\infty}
   \leq\frac{1}{t^{\alpha}}\int^{t}_{0}\left\|\frac{\partial P_{s}f}{\partial s}\right\|_{\infty}ds
 \leq\frac{1}{\alpha}\sup_{s>0}\left\{\frac{1}{s^{\alpha-1}}\left\|\frac{\partial P_{s}f}{\partial s}\right\|_{\infty}\right\}.
\end{align*}
Taking the supremum over $t>0$ on the left-hand side of the above inequality, we get
 \begin{align}\label{242}
  \left\|f\right\|_{\mathcal L_{\alpha}^{c}(\mathcal P)}\leq\frac{1}{\alpha}\sup_{t>0}\left\{\frac{1}{t^{\alpha-1}}\left\|\frac{\partial P_{t}f}{\partial t}\right\|_{\infty}\right\}+\|f\|_{\infty}\leq\frac{1}{\alpha}\left\|f\right\|_{\Lambda_{\alpha}({\mathcal P)}}.
 \end{align}

 We now consider (\ref{b2}) for all $\alpha\geq1$. Let $t>0$.
 By Lemma \ref{531}, we obtain
\begin{align}\label{150}
 &\frac{1}{t^{\alpha}}\left\|P_{t}|(I-P_{t})^{[\alpha]+1}f|^{2}\right\|_{\infty}^{\frac{1}{2}}= \frac{1}{t^{\alpha}}\left\|P_{t}|(I-P_{\frac{t}{2}})^{[\alpha]+1}(I+\sum_{N=1}^{[\alpha]+1}C_{[\alpha]+1}^{N}P_{\frac{Nt}{2}})f|^{2}\right\|_{\infty}^{\frac{1}{2}}
\\ \nonumber\leq&\frac{1}{t^{\alpha}}\left\|P_{t}|(I-P_{\frac{t}{2}})^{[\alpha]+1}f|^{2}\right\|_{\infty}^{\frac{1}{2}}
+\frac{1}{t^{\alpha}}\left\|P_{t}|\sum_{N=1}^{[\alpha]+1}C_{[\alpha]+1}^{N}P_{\frac{Nt}{2}}(I-P_{\frac{t}{2}})^{[\alpha]+1}f|^{2}\right\|_{\infty}^{\frac{1}{2}}
\\ \nonumber=&:I_{1}+I_{2}.
\end{align}
 It is not difficult to derive the following inequality from the monotonicity (\ref{4}):
 \begin{align}\label{qq}
 P_{t}\leq\frac{t}{\frac{t}{2}}P_{\frac{t}{2}}=2P_{\frac{t}{2}}.
 \end{align}
Applying this property to $I_{1}$, we get
\begin{align*}
 I_{1}\leq\frac{\sqrt{2}}{t^{\alpha}}\left\|P_{\frac{t}{2}}|(I-P_{\frac{t}{2}})^{[\alpha]+1}f|^{2}\right\|_{\infty}^{\frac{1}{2}}
 =\frac{\sqrt{2}}{2^{\alpha}}\frac{1}{(\frac{t}{2})^{\alpha}}\left\|P_{\frac{t}{2}}|(I-P_{\frac{t}{2}})^{[\alpha]+1}f|^{2}\right\|_{\infty}^{\frac{1}{2}}
\leq\frac{\sqrt{2}}{2^{\alpha}}\left\|f\right\|_{\dot{\mathcal L}_{\alpha}^{c}(\mathcal P)}.
\end{align*}
Next, we deal with $I_{2}$. By the triangle inequality and the fact that $\|P_{t}f\|_{\infty}\leq \|f\|_{\infty}$, we obtain
\begin{align*}
  I_{2}\leq & \frac{1}{t^{\alpha}}\left\|\sum_{N=1}^{[\alpha]+1}C_{[\alpha]+1}^{N}P_{\frac{Nt}{2}}(I-P_{\frac{t}{2}})^{[\alpha]+1}f\right\|_{\infty}
\\ \leq&\sum_{N=1}^{[\alpha]+1}C_{[\alpha]+1}^{N}\frac{1}{t^{\alpha}}\left\|P_{\frac{Nt}{2}}(I-P_{\frac{t}{2}})^{[\alpha]+1}f\right\|_{\infty}
\\=&\sum_{N=1}^{[\alpha]+1}C_{[\alpha]+1}^{N}\frac{1}{t^{\alpha}}\left\|\int_{0}^{\frac{t}{2}}\frac{\partial P_{s_{1}}}{\partial s_{1}}\cdots\int_{0}^{\frac{t}{2}}\frac{\partial P_{s_{[\alpha]}}}{\partial s_{[\alpha]}} \int_{\frac{Nt}{2}}^{\frac{(N+1)t}{2}}\frac{\partial P_{s_{[\alpha]+1}}f}{\partial s_{[\alpha]+1}}ds_{[\alpha]+1}ds_{[\alpha]}\cdots ds_{1}\right\|_{\infty}.
\end{align*}
It is easy to see that $\frac{\partial P_{s}}{\partial s}\frac{\partial P_{t}}{\partial t}=\frac{\partial^{2} P_{t+s}}{\partial s^{2}}=\frac{\partial^{2} P_{t+s}}{\partial t^{2}}$. Applying this property and letting $v=s_{1}+\cdots+s_{[\alpha]+1}$, we have
\begin{align*}
   I_{2}\leq&\sum_{N=1}^{[\alpha]+1}C_{[\alpha]+1}^{N}\frac{1}{t^{\alpha}}\left\|\int_{0}^{\frac{t}{2}}\cdots\int_{0}^{\frac{t}{2}}\int_{\frac{Nt}{2}}^{\frac{(N+1)t}{2}}\frac{\partial^{[\alpha]+1} P_{s_{1}+s_{2}+\cdots s_{[\alpha]+1}}f}{\partial s_{[\alpha]+1}^{[\alpha]+1}}ds_{[\alpha]+1}ds_{[\alpha]}\cdots ds_{1}\right\|_{\infty}
   \\  \nonumber=&\sum_{N=1}^{[\alpha]+1}C_{[\alpha]+1}^{N}\frac{1}{t^{\alpha}}\left\|\int_{0}^{\frac{t}{2}}\cdots\int_{0}^{\frac{t}{2}} \int_{\frac{Nt}{2}+s_{1}+\cdots+s_{[\alpha]}}^{\frac{(N+1)t}{2}+s_{1}+\cdots+s_{[\alpha]}}\frac{\partial^{[\alpha]+1} P_{v}f}{\partial v^{[\alpha]+1}}dvds_{[\alpha]}\cdots ds_{1}\right\|_{\infty}
   \\ \nonumber\leq&\sum_{N=1}^{[\alpha]+1}C_{[\alpha]+1}^{N}\frac{1}{t^{\alpha}}\int_{0}^{\frac{t}{2}}\cdots\int_{0}^{\frac{t}{2}} \int_{\frac{Nt}{2}+s_{1}+\cdots+s_{[\alpha]}}^{\frac{(N+1)t}{2}+s_{1}+\cdots+s_{[\alpha]}}\frac{v^{\alpha-([\alpha]+1)}}{v^{\alpha-([\alpha]+1)}}\left\|\frac{\partial^{[\alpha]+1} P_{v}f}{\partial v^{[\alpha]+1}}\right\|_{\infty}dvds_{[\alpha]}\cdots ds_{1}
   \\ \nonumber \leq&\left\|f\right\|_{\Lambda_{\alpha}(\mathcal P)}\sum_{N=1}^{[\alpha]+1}C_{[\alpha]+1}^{N}\frac{1}{t^{\alpha}}\int_{0}^{\frac{t}{2}}\cdots\int_{0}^{\frac{t}{2}} \int_{\frac{Nt}{2}+s_{1}+\cdots+s_{[\alpha]}}^{\frac{(N+1)t}{2}+s_{1}+\cdots+s_{[\alpha]}}{v^{\alpha-([\alpha]+1)}}dvds_{[\alpha]}\cdots ds_{1}.
\end{align*}
If $\alpha$ is not an integer, we get
\begin{align}\label{f1}
  &\int_{0}^{\frac{t}{2}}\cdots\int_{0}^{\frac{t}{2}} \int_{\frac{Nt}{2}+s_{1}+\cdots+s_{[\alpha]}}^{\frac{(N+1)t}{2}+s_{1}+\cdots+s_{[\alpha]}}{v^{\alpha-([\alpha]+1)}}dvds_{[\alpha]}\cdots ds_{1}
\\ \nonumber \leq&\int_{0}^{\frac{t}{2}}\cdots\int_{0}^{\frac{t}{2}} \frac{(\frac{Nt}{2}+s_{1}+\cdots+s_{[\alpha]})^{\alpha-[\alpha]}}{\alpha-[\alpha]}ds_{[\alpha]}\cdots ds_{1}
\\ \nonumber \leq&\frac{(\frac{N+1}{2}t)^{\alpha-[\alpha]}}{(\alpha-[\alpha])}\int_{0}^{\frac{t}{2}}\cdots\int_{0}^{\frac{t}{2}} ds_{[\alpha]}\cdots ds_{1}=\frac{(N+1)^{\alpha-[\alpha]}t^{\alpha}}{2^{\alpha}(\alpha-[\alpha])}.
\end{align}
If $\alpha$ is an integer, we may compute
\begin{align}\label{162}
 &\int_{0}^{\frac{t}{2}}\cdots\int_{0}^{\frac{t}{2}} \int_{\frac{Nt}{2}+s_{1}+\cdots+s_{[\alpha]}}^{\frac{(N+1)t}{2}+s_{1}+\cdots+s_{[\alpha]}}{v^{\alpha-([\alpha]+1)}}dvds_{[\alpha]}\cdots ds_{1}
  \\ \nonumber=&\int_{0}^{\frac{t}{2}}\cdots\int_{0}^{\frac{t}{2}} \int_{\frac{Nt}{2}+s_{1}+\cdots+s_{[\alpha]}}^{\frac{(N+1)t}{2}+s_{1}+\cdots+s_{[\alpha]}}{v^{-1}}dvds_{\alpha}\cdots ds_{1}
  \\ \nonumber \leq &\int_{0}^{\frac{t}{2}} \cdots\int_{0}^{\frac{t}{2}} \frac{\frac{t}{2}} {\frac{Nt}{2}+s_{1}+\cdots+s_{\alpha}}ds_{\alpha}\cdots ds_{1}
 \\ \nonumber \leq &\int_{0}^{\frac{t}{2}} \frac{(\frac{t}{2})^{\alpha}} {\frac{Nt}{2}+s_{1}}ds_{1}\leq \frac{t^{\alpha}}{2^{\alpha}N}.
\end{align}
By (\ref{f1}) and (\ref{162}), we obtain
\begin{align}\label{152}
  I_{2}\lesssim&_{\alpha}\left\|f\right\|_{\dot{\Lambda}_{\alpha}(\mathcal P)}.
\end{align}
Taking the supremum over $t>0$ on the left-hand side of \eqref{150} and combining the estimates of $I_1$ and $I_2$, we arrive at
$$\|f\|_{\dot{\mathcal L}_{\alpha}^{c}(\mathcal P)}\leq \frac{\sqrt 2}{2^\alpha}\|f\|_{\dot{\mathcal L}_{\alpha}^{c}(\mathcal P)}+C_\alpha\|f\|_{\dot{\Lambda}_{\alpha}(\mathcal P)}$$
for some $C_\alpha>0$. Then by the fact that $\frac{\sqrt{2}}{2^{\alpha}}<1$ holds for all $\alpha\geq1$, one gets
$$\left\|f\right\|_{\dot{\mathcal L}_{\alpha}^{c}(\mathcal P)}\lesssim_{\alpha}\left\|f\right\|_{\dot{\Lambda}_{\alpha}(\mathcal P)}.$$
This implies
\begin{align}\label{f3}
 \left\|f\right\|_{\mathcal L_{\alpha}^{c}(\mathcal P)} =\|f\|_{\infty}+ \left\|f\right\|_{\dot{\mathcal L}_{\alpha}^{c}(\mathcal P)}\lesssim_{\alpha}\|f\|_{\infty}+\left\|f\right\|_{\dot{\Lambda}_{\alpha}(\mathcal P)}=\left\|f\right\|_{\Lambda_{\alpha}(\mathcal P)}.
\end{align}
Combining the estimates of (\ref{242}) and (\ref{f3}), we get the assertion (\ref{b2}).

 \bigskip

We now prove (\ref{192}): for $0<\alpha<2$,
\begin{align*}
 \left\|f\right\|_{\Lambda_{\alpha}(\mathcal P)}\lesssim_{\alpha}\left\|f\right\|_{\mathcal L_{\alpha}^{c}(\mathcal P)}.
\end{align*}
Similar to Step 1 in the proof of Proposition \ref{20}, it suffices to show that for all $f\in\mathcal M^{\circ}$,
\begin{align}\label{h7}
 \left\|f\right\|_{\dot{\Lambda}_{\alpha}({\mathcal P})}\lesssim_{\alpha} \left\|f\right\|_{\dot{\mathcal L}_{\alpha}^{c}(\mathcal P)}.
\end{align}
Fix $f\in \mathcal{M}^{\circ}$.

{\it Case 1.} $0<\alpha<1$.
Since $P_{t}f\rightarrow 0$ as $t\rightarrow \infty$, we can split $f$ into two parts for any $t>0$:
 \begin{align}\label{cd}
 f=f-P_{t}f+\sum_{m\geq0} (P_{2^{m}t}-P_{_{2^{m+1}t}})f.
 \end{align}
 For $t>0$ fixed, by $\frac{\partial P_{t}}{\partial t}P_{2^{m}t}=P_{t}\left.\frac{\partial P_{s}}{\partial s}\right|_{s=2^{m}t}$, we have
\begin{align*}
 \frac{1}{t^{\alpha-1}}\left\|\frac{\partial P_{t}f}{\partial t}\right\|_{\infty}
  \leq&\frac{1}{t^{\alpha-1}}\left\|\frac{\partial P_{t}}{\partial t}(f-P_{t}f)\right\|_{\infty}+\frac{1}{t^{\alpha-1}}\left\|\frac{\partial P_{t}}{\partial t}\sum_{m\geq0}(P_{2^{m}t}f-P_{2^{m+1}t}f)\right\|_{\infty}
  \\ \nonumber \leq&\frac{1}{t^{\alpha-1}}\left\|\frac{\partial P_{t}}{\partial t}(f-P_{t}f)\right\|_{\infty}+\sum_{m\geq0}\frac{1}{t^{\alpha-1}}\left\|P_{t}\left.\frac{\partial P_{s}}{\partial s}\right|_{s=2^{m}t}(f-P_{2^{m}t}f)\right\|_{\infty}
   \\ \nonumber \leq&\frac{1}{t^{\alpha-1}}\left\|\frac{\partial P_{t}}{\partial t}(f-P_{t}f)\right\|_{\infty}+\sum_{m\geq0}\frac{1}{t^{\alpha-1}}\left\|\left.\frac{\partial P_{s}}{\partial s}\right|_{s=2^{m}t}(f-P_{2^{m}t}f)\right\|_{\infty}.
\end{align*}
With the help of Proposition \ref{ws}, we get
\begin{align}\label{241}
 \frac{1}{t^{\alpha-1}}\left\|\frac{\partial P_{t}f}{\partial t}\right\|_{\infty}\leq &\frac{1}{t^{\alpha-1}}\left\|\frac{\partial P_{t}}{\partial t}(f-P_{t}f)\right\|_{\infty}+\sum_{m\geq0}\frac{1}{t^{\alpha-1}}\left\|\left.\frac{\partial P_{s}}{\partial s}\right|_{s=2^{m}t}(f-P_{2^{m}t}f)\right\|_{\infty}
   \\ \nonumber =&\frac{1}{t^{\alpha-1}}\left\|\left|\frac{\partial P_{t}}{\partial t}(f-P_{t}f)\right|^{2}\right\|^{\frac{1}{2}}_{\infty}+\sum_{m\geq0}\frac{1}{t^{\alpha-1}}\left\|\left|\left.\frac{\partial P_{s}}{\partial s}\right|_{s=2^{m}t}(f-P_{2^{m}t}f)\right|^{2}\right\|^{\frac{1}{2}}_{\infty}
  \\ \nonumber \lesssim&
  \frac{1}{t^{\alpha}}\left\| P_{t} \left|f-P_{t}f\right|^{2}\right\|^{\frac{1}{2}}_{\infty}+\sum_{m\geq0}\frac{1}{2^{m}}\frac{2^{m\alpha}}{(2^{m}t)^{\alpha}}\left\| P_{2^{m}t}\left|f-P_{2^{m}t}f\right|^{2}\right\|^{\frac{1}{2}}_{\infty}
  \\ \nonumber \leq&(1+\sum_{m\geq0}2^{m(\alpha-1)})\sup_{t>0}\frac{1}{t^{\alpha}}\left\| P_{t} \left|f-P_{t}f\right|^{2}\right\|^{\frac{1}{2}}_{\infty}.
\end{align}
It is not difficult to find that the series $\sum_{m\geq0}2^{m(\alpha-1)}$ converges for any $\alpha<1$. Then taking the supremum over $t>0$ on the left-hand side of (\ref{241}),  one gets
\begin{align}\label{15}
  \left\|f\right\|_{\dot{\Lambda}_{\alpha}({\mathcal P)}}\lesssim&_{\alpha}\left\|f\right\|_{\dot{\mathcal L}_{\alpha}^{c}(\mathcal P)}.
\end{align}

{\it Case 2.} $1\leq\alpha<2$.
 Note that, for $t>0$ fixed, $(I-P_{t})f$ can be written as
 \begin{align}\label{cd}
 (I-P_{t})f&=(I-P_{t})^{2}f+\sum_{n\geq1} (I-P_{t})(P_{nt}-P_{(n+1)t})f\\
 \nonumber&=(I-P_{t})^{2}f+\sum_{n\geq1} P_{nt}(I-P_{t})^{2}f.
 \end{align}
Applying this property and Proposition \ref{ws}, we have
\begin{align*}
\frac{1}{t^{\alpha-2}}\left\|\frac{\partial^{2}P_{t}}{\partial t^{2}}(I-P_{t})f\right\|_{\infty}
  \leq&\frac{1}{t^{\alpha-2}}\left\|\frac{\partial^{2}P_{t}}{\partial t^{2}}(I-P_{t})^{2}f\right\|_{\infty}
+\sum_{n\geq1}\frac{1}{t^{\alpha-2}}\left\|\frac{\partial^{2}P_{t}}{\partial t^{2}}P_{nt}(I-P_{t})^{2}f\right\|_{\infty}
\\ =&\frac{1}{t^{\alpha-2}}\left\|\left|\frac{\partial^{2}P_{t}}{\partial t^{2}}(I-P_{t})^{2}f\right|^{2}\right\|_{\infty}^{\frac{1}{2}}
+\sum_{n\geq1}\frac{1}{t^{\alpha-2}}\left\|\left|\left.\frac{\partial^{2}P_{s}}{\partial s^{2}}\right|_{s=(n+1)t}(I-P_{t})^{2}f\right|^{2}\right\|_{\infty}^{\frac{1}{2}}
 \\ \lesssim &\frac{1}{t^{\alpha}}\left\|P_{t}|(I-P_{t})^{2}f|^{2}\right\|_{\infty}^{\frac{1}{2}}
+\sum_{n\geq1}\frac{1}{(n+1)^{2}}\frac{1}{t^{\alpha}}\left\|P_{(n+1)t}|(I-P_{t})^{2}f|^{2}\right\|_{\infty}^{\frac{1}{2}}
\\(\ref{4})\leq&\frac{1}{t^{\alpha}}\left\|P_{t}|(I-P_{t})^{2}f|^{2}\right\|_{\infty}^{\frac{1}{2}}
+\sum_{n\geq1}\frac{1}{(n+1)^{2}}\frac{1}{t^{\alpha}}\left\|(n+1)P_{t}|(I-P_{t})^{2}f|^{2}\right\|_{\infty}^{\frac{1}{2}}
\\ \nonumber=&(1+\sum_{n\geq1}\frac{1}{(n+1)^{\frac{3}{2}}})\frac{1}{t^{\alpha}}\left\|P_{t}|(I-P_{t})^{2}f|^{2}\right\|_{\infty}^{\frac{1}{2}}.
\end{align*}
Thus, taking the supremum over $t>0$ on both sides and using Lemma \ref{190} (ii), for $1\leq\alpha<2$, we obtain
\begin{align}\label{196}
 \nonumber \left\|f\right\|_{\dot{\Lambda}_{\alpha}({\mathcal P)}}\leq&\frac{1}{1-2^{\alpha-([ \alpha]+1)}}\sup_{t>0}\frac{1}{t^{\alpha-2}}\left\|\frac{\partial^{2}P_{t}}{\partial t^{2}}(f-P_{t}f)\right\|_{\infty}
  \\ \nonumber\leq&\frac{1}{1-2^{\alpha-([ \alpha]+1)}}(1+\sum_{n\geq1}\frac{1}{(n+1)^{\frac{3}{2}}})\sup_{t>0}\frac{1}{t^{\alpha}}\left\|P_{t}|(I-P_{t})^{2}f|^{2}\right\|_{\infty}^{\frac{1}{2}}
  \\ \lesssim&\frac{1}{1-2^{\alpha-([ \alpha]+1)}}\|f\|_{\mathcal{\dot{L}}^{c}_{\alpha}(\mathcal{P})}.
\end{align}
In the last inequality, we have used the convergence of the series $\sum_{n\geq1}\frac{1}{(n+1)^{\frac{3}{2}}}$. Therefore, $\mathcal L_{\alpha}^{c}(\mathcal P)=\Lambda_{\alpha}({\mathcal P})$ with equivalent norms when $0<\alpha<2$. Taking adjoints, we see that $\mathcal L_{\alpha}^{r}(\mathcal P)$ is isomorphic to $\Lambda_{\alpha}({\mathcal P})$ when $0<\alpha<2$.
The proof of Theorem \ref{e25} is complete.
\end{proof}

\begin{remark} From the constant appearing in \eqref{196}, it is natural to conjecture that Theorem \ref{e25} should hold for any $\alpha\geq2$. But at the moment of writing, we are unable to verify it.
\end{remark}
\section{The proof of Theorem \ref{22}}
\hskip\parindent
The goal of this section is to show Theorem \ref{22}, which says that the elements in Campanato spaces associated with the Poisson semigroup $(P_{t})_{t\geq0}$ enjoys automatically the higher order cancellation property. This leads to a new perspective even in the commutative setting, and one may find further applications in the study of spectral multipliers like the ones in \cite{ cdly20, DSY}.

Let $(T_{t})_{t\geq0}$ be a Markov semigroup and $(P_{t})_{t\geq0}$ be its subordinated Poisson semigroup. Suppose that $k$ is an integer greater than $\alpha>0$. We define
\begin{align*}
  \mathcal L_{\alpha,k}^{c}(\mathcal P)=\left\{f\in \mathcal M: \left\|f\right\|_{\mathcal L_{\alpha,k}^{c}(\mathcal P)}<\infty\right\},
\end{align*}
where
\begin{align}\label{264}
 \left\|f\right\|_{\mathcal L_{\alpha,k}^{c}(\mathcal P)}=\|f\|_{\infty}+\left\|f\right\|_{\dot{\mathcal L}_{\alpha,k}^{c}(\mathcal P)},\quad \left\|f\right\|_{\dot{\mathcal L}_{\alpha,k}^{c}(\mathcal P)}=\sup_{t>0}\frac{1}{t^{\alpha}}
   \left\|P_{t}|(I-P_{t})^{k}f|^{2}\right\|_{\infty}^{\frac{1}{2}}.
\end{align}
As in Section 4, one can easily verify that $\left\|\cdot\right\|_{\mathcal L_{\alpha,k}^{c}(\mathcal P)}$ is a norm and $\mathcal L_{\alpha,k}^{c}(\mathcal P)$ is complete with respect to the norm $\|\cdot\|_{\mathcal L_{\alpha,k}^{c}(\mathcal P)}$.
The row space $\mathcal L_{\alpha,k}^{r}(\mathcal P)$ is the space of all $f$ such that $f^{*}\in\mathcal L_{\alpha,k}^{c}(\mathcal P)$, equipped with the norm $\left\|f\right\|_{\mathcal L_{\alpha,k}^{r}(\mathcal P)}=\left\|f^{*}\right\|_{\mathcal L_{\alpha,k}^{c}(\mathcal P)}$, and the symmetric space $\mathcal L_{\alpha,k}^{cr}(\mathcal P)$ is defined as the intersection of $\mathcal L_{\alpha,k}^{c}(\mathcal P)$ and $\mathcal L_{\alpha,k}^{r}(\mathcal P)$.
\begin{lemma}\label{j1j}
 Let $0<\alpha\leq4$. The function $F(\alpha,\cdot)$ defined on the integers
 $$F(\alpha,m)=\sum_{r=1}^{m-1}C_{m-1}^{r}(r+1)^{\alpha-m},\quad m\geq2$$
 satisfies
\begin{align*}
 \max_{m\geq2}F(\alpha,m)\leq F(\alpha,3).
\end{align*}

\end{lemma}
\begin{proof}
Let $\alpha>0$.
For $m=2$, one can easily check that $F(\alpha,2)\leq F(\alpha,3)$, since $F(\alpha,2)=2^{\alpha-2}$ and $F(\alpha,3)=2^{\alpha-2}+3^{\alpha-3}$.
 For $m\geq3$, we will conclude the desired result by showing that $F(\alpha,m)$ is decreasing with respect to $m$.
By simple computation,
\begin{align*}
  F(\alpha,m)-F(\alpha,m+1)=& \sum_{r=1}^{m-1}C_{m-1}^{r}(r+1)^{\alpha-m}-\sum_{r=1}^{m}C_{m}^{r}(r+1)^{\alpha-(m+1)}                   \\=&\sum_{r=1}^{m-1}(C_{m-1}^{r}-\frac{1}{r+1}C_{m}^{r})(r+1)^{\alpha-m}-C_{m}^{m}(m+1)^{\alpha-(m+1)}
\\=&\sum_{r=1}^{m-1}\frac{(m-1)!r(m-r-1)}{r!(m-r)!(r+1)^{m+1-\alpha}}-\frac{1}{(m+1)^{m+1-\alpha}}.
\end{align*}
Noting that all the terms in the above summation is non-negative, and thus the summation is not smaller than the first term $r=1$. Therefore, by noting $m\geq3\geq\alpha-1$,

\begin{align*}
F(\alpha,m)-F(\alpha,m+1)\geq& \frac{m-2}{2^{m+1-\alpha}}-\frac{1}{(m+1)^{m+1-\alpha}}\\
\geq&\frac{1}{2^{m+1-\alpha}}-\frac{1}{2^{m+1-\alpha}}\geq0.
 \end{align*}
\end{proof}

To prove Theorem \ref{22}, we first need the following variant of Theorem \ref{e25}.
\begin{theorem}\label{261}
Suppose $f\in\mathcal{M}$. Then
\begin{itemize}
  \item [\rm{(i)}] for all $\alpha>0$ and all integer $k>\alpha$, $\left\|f\right\|_{\mathcal L_{\alpha,k}^{c}(\mathcal P)}\lesssim_{\alpha,k}\left\|f\right\|_{\Lambda_{\alpha,k}(\mathcal P)};$
  \item  [\rm{(ii)}] for all $0<\alpha<\alpha_0$ where $\alpha_0$ verifies $2^{\alpha_{0}-2}+3^{\alpha_{0}-3}=1$ and all integer $k>\alpha$, $\left\|f\right\|_{\Lambda_{\alpha,k}(\mathcal P)}\lesssim_{\alpha,k}\left\|f\right\|_{\mathcal L_{\alpha,k}^{c}(\mathcal P)}$.
\end{itemize}
\end{theorem}
\begin{proof}
Note that the situation $k=[\alpha]+1$ has been discussed in Theorem \ref{e25}. Thus, in the rest of this proof, we just need to consider the case of $k\geq[\alpha]+2$.

 $\rm{(i)}$. We start with the case of $0<\alpha<1$. Similar to the proof of \eqref{b2}, for $t>0$ fixed, we get
\begin{align*}
  \frac{1}{t^{\alpha}}\left\|P_{t}|(I-P_{t})^{k}f|^{2}\right\|_{\infty}^{\frac{1}{2}}\leq&\frac{1}{t^{\alpha}}\left\|(I-P_{t})^{k}f\right\|_{\infty}
\\=&\frac{1}{t^{\alpha}}\left\|\int^{t}_{0}\cdots\int^{t}_{0}\int^{t}_{0}\frac{\partial^{k}P_{s_{1}+\cdots +s_{k-1}+s_{k}}f}{\partial s_{k}^{k}}ds_{k}ds_{k-1}\cdots ds_{1}\right\|_{\infty}
\\(v=s_{1}+\cdots +s_{k-1}+s_{k})\leq&\frac{1}{t^{\alpha}}\int^{t}_{0}\cdots\int^{t}_{0}\int^{t+s_{1}+\cdots +s_{k-1}}_{s_{1}+\cdots +s_{k-1}}\left\|\frac{\partial^{k}P_{v}f}{\partial v^{k}}\right\|_{\infty}dvds_{k-1}\cdots ds_{1}
\\ \leq&\left\|f\right\|_{\dot{\Lambda}_{\alpha,k}(\mathcal P)}\frac{1}{t^{\alpha}}\int^{t}_{0}\cdots\int^{t}_{0}\int^{t+s_{1}+\cdots +s_{k-1}}_{s_{1}+\cdots +s_{k-1}}v^{\alpha-k}dvds_{k-1}\cdots ds_{1}.
\end{align*}
Note that
$$\alpha-k<\alpha+1-k<\cdots<\alpha+(k-2)-k<-1,\;-1<\alpha-1<0.$$
We then obtain
\begin{align*}
    \frac{1}{t^{\alpha}}\left\|P_{t}|(I-P_{t})^{k}f|^{2}\right\|_{\infty}^{\frac{1}{2}}
    \leq&\left\|f\right\|_{\dot{\Lambda}_{\alpha,k}(\mathcal P)}\frac{1}{t^{\alpha}}\int^{t}_{0}\cdots\int^{t}_{0}\frac{(s_{1}+\cdots +s_{k-1})^{\alpha+1-k}}{k-(\alpha+1)}ds_{k-1}\cdots ds_{1}
\\ \leq&\left\|f\right\|_{\dot{\Lambda}_{\alpha,k}(\mathcal P)}\frac{1}{t^{\alpha}}\int^{t}_{0}\frac{s_{1}^{\alpha-1}}{(k-(\alpha+1))\cdots(1-\alpha)}ds_{k-1}\cdots ds_{1}
\\ =&\frac{1}{(k-(\alpha+1))\cdots(1-\alpha)\alpha}\left\|f\right\|_{\dot{\Lambda}_{\alpha,k}(\mathcal P)}.
\end{align*}
Taking the supremum over $t>0$ on the left-hand side, we get
\begin{align*}
 \left\|f\right\|_{\dot{\mathcal L}_{\alpha,k}^{c}(\mathcal P)}\leq\frac{1}{(k-(\alpha+1))(k-(\alpha+2))\cdots(1-\alpha)\alpha}\left\|f\right\|_{\dot{\Lambda}_{\alpha,k}(\mathcal P)},
\end{align*}
which implies
\begin{align}\label{x1}
\left\|f\right\|_{\mathcal L_{\alpha,k}^{c}(\mathcal P)}\leq& \|f\|_{\infty}+\frac{1}{(k-(\alpha+1))\cdots(1-\alpha)\alpha}\left\|f\right\|_{\dot{\Lambda}_{\alpha,k}(\mathcal P)}
\\ \nonumber \leq&\max\{1,\frac{1}{(k-(\alpha+1))\cdots(1-\alpha)\alpha}\}\left\|f\right\|_{\Lambda_{\alpha,k}(\mathcal P)}.
\end{align}

We now consider the situation $\alpha\geq1$. By Lemma \ref{531} and \eqref{qq}, for $t>0$ fixed, we have
\begin{align}\label{h}
 \frac{1}{t^{\alpha}}\left\|P_{t}|(I-P_{t})^{k}f|^{2}\right\|_{\infty}^{\frac{1}{2}}
=&  \frac{1}{t^{\alpha}}\left\|P_{t}\left|(I-P_{\frac{t}{2}})^{k}(I+\sum_{r=1}^{k}C_{k}^{r}P_{\frac{r}{2}t})f\right|^{2}\right\|_{\infty}^{\frac{1}{2}}
\\ \nonumber
   \leq&\frac{1}{t^{\alpha}}\left\|P_{t}\left|(I-P_{\frac{t}{2}})^{k}f\right|^{2}\right\|_{\infty}^{\frac{1}{2}}
+\sum_{r=1}^{k}C_{k}^{r}\frac{1}{t^{\alpha}}
  \left\|P_{t}\left|P_{\frac{rt}{2}}(I-P_{\frac{t}{2}})^{k}f\right|^{2}\right\|_{\infty}^{\frac{1}{2}}
\\ \nonumber \leq&
\frac{\sqrt{2}}{2^{\alpha}}\frac{1}{(\frac{t}{2})^{\alpha}}\left\|P_{\frac{t}{2}}\left|(I-P_{\frac{t}{2}})^{k}f\right|^{2}\right\|_{\infty}^{\frac{1}{2}}
+\sum_{r=1}^{k}C_{k}^{r}\frac{1}{t^{\alpha}}
  \left\|P_{t}\left|P_{\frac{rt}{2}}(I-P_{\frac{t}{2}})^{k}f\right|^{2}\right\|_{\infty}^{\frac{1}{2}}
\\ \nonumber =:& J_{1}+J_{2}.
\end{align}
For $J_{2}$, by \eqref{5151}, we deduce
\begin{align*}
 J_{2}\leq&\sum_{r=1}^{k}C_{k}^{r}\frac{1}{t^{\alpha}}
  \left\|P_{\frac{rt}{2}}(I-P_{\frac{t}{2}})^{k}f\right\|_{\infty}
\\  =&\sum_{r=1}^{k}C_{k}^{r}\frac{1}{t^{\alpha}}\left\|\int^{\frac{t}{2}}_{0}\cdots\int^{\frac{t}{2}}_{0}\int^{\frac{(r+1)t}{2}}_{\frac{rt}{2}}\frac{\partial^{k}P_{s_{1}+\cdots +s_{k-1}+s_{k}}f}{\partial s_{k}^{k}}ds_{k}ds_{k-1}\cdots ds_{1}\right\|_{\infty}.
\end{align*}
Let $s=s_{1}+\cdots +s_{k-1}+s_{k}$. One gets
\begin{align*}
 J_{2}\leq&
 \sum_{r=1}^{k}C_{k}^{r}\frac{1}{t^{\alpha}}\int^{\frac{t}{2}}_{0}\cdots\int^{\frac{t}{2}}_{0}\int^{\frac{(r+1)t}{2}+s_{1}+\cdots +s_{k-1}}_{\frac{rt}{2}+s_{1}+\cdots +s_{k-1}}\left\|\frac{\partial^{k}P_{s}f}{\partial s^{k}}\right\|_{\infty}dsds_{k-1}\cdots ds_{1}
 \\ \leq&\left\|f\right\|_{\dot{\Lambda}_{\alpha,k}(\mathcal P)}\sum_{r=1}^{k}C_{k}^{r}\frac{1}{t^{\alpha}}\int^{\frac{t}{2}}_{0}\cdots\int^{\frac{t}{2}}_{0}\int^{\frac{(r+1)t}{2}+s_{1}+\cdots +s_{k-1}}_{\frac{rt}{2}+s_{1}+\cdots +s_{k-1}}s^{\alpha-k}dsds_{k-1}\cdots ds_{1}
  \\ \leq&\left\|f\right\|_{\dot{\Lambda}_{\alpha,k}(\mathcal P)}\sum_{r=1}^{k}C_{k}^{r}\frac{1}{t^{\alpha}}\int^{\frac{t}{2}}_{0}\cdots\int^{\frac{t}{2}}_{0}\frac{(\frac{rt}{2}+s_{1}+\cdots +s_{k-1})^{\alpha+1-k}}{k-(\alpha+1)}ds_{k-1}\cdots ds_{1}
  \\ \leq&\sum_{r=1}^{k}C_{k}^{r}\frac{r^{\alpha+1-k}}{2^{\alpha}(k-(\alpha+1))}\left\|f\right\|_{\dot{\Lambda}_{\alpha,k}(\mathcal P)}.
\end{align*}
Taking the supremum over $t>0$ on both sides of \eqref{h}, we obtain
\begin{align*}
  \left\|f\right\|_{\dot{\mathcal L}_{\alpha,k}^{c}(\mathcal P)}\leq\frac{\sqrt{2}}{2^{\alpha}}\left\|f\right\|_{\dot{\mathcal L}_{\alpha,k}^{c}(\mathcal P)}+\sum_{r=1}^{k}C_{k}^{r}\frac{r^{\alpha+1-k}}{2^{\alpha}(k-(\alpha+1))}\left\|f\right\|_{\dot{\Lambda}_{\alpha,k}(\mathcal P)}.
\end{align*}
Since $\frac{\sqrt{2}}{2^{\alpha}}<1$ holds for any $\alpha\geq1$. Thus, we have
$$\left\|f\right\|_{\dot{\mathcal L}_{\alpha,k}^{c}(\mathcal P)}\lesssim_{\alpha,k}\left\|f\right\|_{\dot{\Lambda}_{\alpha,k}(\mathcal P)}.$$
This implies
\begin{align}\label{x2}
\left\|f\right\|_{\mathcal L_{\alpha,k}^{c}(\mathcal P)}=&\|f\|_{\infty}+\left\|f\right\|_{\dot{\mathcal L}_{\alpha,k}^{c}(\mathcal P)}
\lesssim_{\alpha,k} \|f\|_{\infty}+\left\|f\right\|_{\dot{\Lambda}_{\alpha,k}(\mathcal P)}
=\left\|f\right\|_{\Lambda_{\alpha,k}(\mathcal P)}.
\end{align}
Integrating \eqref{x1} with \eqref{x2}, we obtain (\rm{i}).

\bigskip

\rm{(ii)}.
Similar to the proof of Theorem \ref{e25}, we just need to show for $f\in\mathcal M^{\circ}$,
\begin{align}\label{x5}
  \left\|f\right\|_{\dot{\Lambda}_{\alpha,k}(\mathcal P)}\lesssim_{\alpha,k}\left\|f\right\|_{\dot{\mathcal L}_{\alpha,k}^{c}(\mathcal P)}.
\end{align}
Let $0<\alpha<\alpha_0$ and $k\geq[\alpha]+2$. By (\ref{cd}), for a fixed $f\in\mathcal M^{\circ}$ and $t>0$, one has
\begin{align*}
   (I-P_{t})^{k-1}f=&(I-P_{t})^{k}f+\sum_{n\geq1} (I-P_{t})^{k-1}(P_{nt}-P_{(n+1)t})f
   \\ \nonumber =&(I-P_{t})^{k}f+\sum_{n\geq1} P_{nt}(I-P_{t})^{k}f.
\end{align*}
Besides, one may compute
$$I=(I-P_{t})^{k-1}+I-(I-P_{t})^{k-1}
  =(I-P_{t})^{k-1}+\sum_{r=1}^{k-1}(-1)^{r+1}C_{k-1}^{r}P_{rt}.$$
Applying these two properties and \eqref{a1}, for $t>0$ fixed, we then have
\begin{align*}
  \frac{1}{t^{\alpha-k}}\left\|\frac{\partial^{k} P_{t}f}{\partial t^{k}}\right\|_{\infty}
  \leq&  \frac{1}{t^{\alpha-k}}\left\|\frac{\partial^{k} P_{t}}{\partial t^{k}}(I-P_{t})^{k-1}f\right\|_{\infty}
+\sum_{r=1}^{k-1}C_{k-1}^{r}\frac{1}{t^{\alpha-k}}\left\|\frac{\partial^{k} P_{t}}{\partial t^{k}}P_{rt}f\right\|_{\infty}
\\ \leq&  \frac{1}{t^{\alpha-k}}\left\|\frac{\partial^{k} P_{t}}{\partial t^{k}}(I-P_{t})^{k}f\right\|_{\infty}
+\sum_{n\geq1}  \frac{1}{t^{\alpha-k}}\left\|\frac{\partial^{k} P_{t}}{\partial t^{k}}P_{nt}(I-P_{t})^{k}f\right\|_{\infty}
\\&+\sum_{r=1}^{k-1}C_{k-1}^{r}\frac{1}{t^{\alpha-k}}\left\|\frac{\partial^{k} P_{t}}{\partial t^{k}}P_{rt}f\right\|_{\infty}
\\ =&  \frac{1}{t^{\alpha-k}}\left\|\left|\frac{\partial^{k} P_{t}}{\partial t^{k}}(I-P_{t})^{k}f\right|^{2}\right\|_{\infty}^{\frac{1}{2}}
+\sum_{n\geq1}  \frac{1}{t^{\alpha-k}}\left\|\left|\left.\frac{\partial^{k} P_{s}}{\partial s^{k}}\right|_{(n+1)t}(I-P_{t})^{k}f\right|^{2}\right\|_{\infty}^{\frac{1}{2}}
\\&+\sum_{r=1}^{k-1}C_{k-1}^{r}\frac{(r+1)^{\alpha-k}}{((r+1)t)^{\alpha-k}}\left\|\left.\frac{\partial^{k} P_{v}}{\partial v^{k}}\right|_{(r+1)t}f\right\|_{\infty}
\\ \leq& 2^{\frac{k(k+4)}{2}}\frac{1}{t^{\alpha}}\left\| P_{t}\left|(I-P_{t})^{k}f\right|^{2}\right\|_{\infty}^{\frac{1}{2}}
\\&+2^{\frac{k(k+4)}{2}}\sum_{n\geq1}\frac{1}{(n+1)^{k}}  \frac{1}{t^{\alpha}}\left\|P_{(n+1)t}\left|(I-P_{t})^{k}f\right|^{2}\right\|_{\infty}^{\frac{1}{2}}
\\&+\sum_{r=1}^{k-1}C_{k-1}^{r}(r+1)^{\alpha-k}\left\|f\right\|_{\dot{\Lambda}_{\alpha,k}(\mathcal P)}.
\end{align*}
Note that $\sum_{n\geq1}\frac{1}{(n+1)^{k-\frac{1}{2}}}$ is a convergent series since $k\geq2$. By \eqref{4}, we have
\begin{align*}
 \frac{1}{t^{\alpha-k}}\left\|\frac{\partial^{k} P_{t}f}{\partial t^{k}}\right\|_{\infty}
  \leq& 2^{\frac{k(k+4)}{2}}\frac{1}{t^{\alpha}}\left\| P_{t}\left|(I-P_{t})^{k}f\right|^{2}\right\|_{\infty}^{\frac{1}{2}}
\\&+2^{\frac{k(k+4)}{2}}\sum_{n\geq1}\frac{1}{(n+1)^{k-\frac{1}{2}}}  \frac{1}{t^{\alpha}}\left\|P_{t}\left|(I-P_{t})^{k}f\right|^{2}\right\|_{\infty}^{\frac{1}{2}}
\\ \nonumber&+\sum_{r=1}^{k-1}C_{k-1}^{r}(r+1)^{\alpha-k}\left\|f\right\|_{\dot{\Lambda}_{\alpha,k}(\mathcal P)}
\\ \nonumber\leq&2^{\frac{k(k+4)}{2}}(1+\sum_{n\geq1}\frac{1}{(n+1)^{k-\frac{1}{2}}}) \left\|f\right\|_{\dot{\mathcal L}_{\alpha,k}^{c}(\mathcal P)}
\\ \nonumber&+\sum_{r=1}^{k-1}C_{k-1}^{r}(r+1)^{\alpha-k}\left\|f\right\|_{\dot{\Lambda}_{\alpha,k}(\mathcal P)}.
\end{align*}
Taking the supremum over $t>0$ on the left-hand side, we get
\begin{align}\label{x6}
  \left\|f\right\|_{\dot{\Lambda}_{\alpha,k}(\mathcal P)}\leq2^{\frac{k(k+4)}{2}}(1+\sum_{n\geq1}\frac{1}{(n+1)^{k-\frac{1}{2}}}) \left\|f\right\|_{\dot{\mathcal L}_{\alpha,k}^{c}(\mathcal P)}
+\sum_{r=1}^{k-1}C_{k-1}^{r}(r+1)^{\alpha-k}\left\|f\right\|_{\dot{\Lambda}_{\alpha,k}(\mathcal P)}.
\end{align}
 By Lemma \ref{j1j}, we have
\begin{align*}
\sum_{r=1}^{k-1}C_{k-1}^{r}(r+1)^{\alpha-k}= F(\alpha,k)\leq F(\alpha,3)=2^{\alpha-2}+3^{\alpha-3}<2^{\alpha_0-2}+3^{\alpha_0-3}=1.
\end{align*}
Here in the last inequality, we have used the fact that $F(\alpha,3)$ is a increasing function with respect to $\alpha>0$.
Therefore, when $0<\alpha<\alpha_{0}$, we obtain
\begin{align*}
   \left\|f\right\|_{\dot{\Lambda}_{\alpha,k}(\mathcal P)}\lesssim\frac{2^{\frac{k(k+4)}{2}}}{1-\sum_{r=1}^{k-1}C_{k-1}^{r}(r+1)^{\alpha-k}}\left\|f\right\|_{\dot{\mathcal L}_{\alpha,k}^{c}(\mathcal P)}
\end{align*}
 for any $k\geq[\alpha]+2$.
 Hence the assertion is completely proved.
\end{proof}
We are now at the position to show Theorem \ref{22}.

\begin{proof}[Proof of Theorem \ref{22}]
Let $\alpha>0$ and $k$ be an integer such that $k\geq[\alpha]+1$. Fix $t>0$,
by the triangle inequality, \eqref{4} and \eqref{1}, we have
\begin{align}\label{532}
  \frac{1}{t^{\alpha}}\left\|P_{t}|(I-P_{t})^{k}f|^{2}\right\|_{\infty}^{\frac{1}{2}}
   =&\frac{1}{t^{\alpha}}\left\|P_{t}|(I-P_{t})^{k-1}(f-P_{t}f)|^{2}\right\|_{\infty}^{\frac{1}{2}}
\\ \nonumber\leq&\frac{1}{t^{\alpha}}\left\|P_{t}|(I-P_{t})^{k-1}f|^{2}\right\|_{\infty}^{\frac{1}{2}}
+\frac{1}{t^{\alpha}}\left\|P_{t}|P_{t}(I-P_{t})^{k-1}f|^{2}\right\|_{\infty}^{\frac{1}{2}}
\\ \nonumber\leq& \left\|f\right\|_{\dot{\mathcal L}_{\alpha,k-1}^{c}(\mathcal P)}
+\frac{1}{t^{\alpha}}\left\|P_{2t}|(I-P_{t})^{k-1}f|^{2}\right\|_{\infty}^{\frac{1}{2}}
\\ \nonumber\leq& \left\|f\right\|_{\dot{\mathcal L}_{\alpha,k-1}^{c}(\mathcal P)}
+\frac{2}{t^{\alpha}}\left\|P_{t}|(I-P_{t})^{k-1}f|^{2}\right\|_{\infty}^{\frac{1}{2}}
\\ \nonumber\leq& 3\left\|f\right\|_{\dot{\mathcal L}_{\alpha,k-1}^{c}(\mathcal P)}
\\ \nonumber \leq& 3^{k-([\alpha]+1)}\left\|f\right\|_{\dot{\mathcal L}_{\alpha}^{c}(\mathcal P)}.
\end{align}
The last two inequalities are true because we have used iteration. Then, taking the supremum over $t>0$ on the left-hand side, we get
\begin{align*}
 \left\|f\right\|_{\dot{\mathcal L}_{\alpha,k}^{c}(\mathcal P)}\leq 3^{k-([\alpha]+1)}\left\|f\right\|_{\dot{\mathcal L}_{\alpha}^{c}(\mathcal P)}.
\end{align*}
This implies
\begin{align}\label{v1}
  \left\|f\right\|_{\mathcal L_{\alpha,k}^{c}(\mathcal P)}
 \leq\|f\|_{\infty}+3^{k-([\alpha]+1)}\left\|f\right\|_{\dot{\mathcal L}_{\alpha}^{c}(\mathcal P)}
 \leq3^{k-([\alpha]+1)}\left\|f\right\|_{\mathcal L_{\alpha}^{c}(\mathcal P)}.
\end{align}
On the other hand, for $0<\alpha<\alpha_{0}$, together with Proposition \ref{20}, Theorem \ref{e25} and Theorem \ref{261}, we deduce
\begin{align}\label{v2}
 \|f\|_{\mathcal L_{\alpha}^{c}(\mathcal P)}\lesssim_{\alpha}\|f\|_{\Lambda_{\alpha}(\mathcal P)}\lesssim_{\alpha,k}\|f\|_{\Lambda_{\alpha,k}(\mathcal P)}\lesssim_{\alpha,k}\|f\|_{\mathcal L_{\alpha,k}^{c}(\mathcal P)}
\end{align}
holds for any $k\geq[\alpha]+1$.
Therefore, combining the estimates of \eqref{v1} and \eqref{v2} yields the assertion.
\end{proof}
Similar to Theorem \ref{b3}, it is easy to get the following corollary by Theorem \ref{261}.
\begin{corollary}
Let $0<\alpha<\alpha_0$ where $2^{\alpha_{0}-2}+3^{\alpha_{0}-3}=1$, and $k$ be an integer such that $k>\alpha$. Then the space $\mathcal L_{\alpha,k}^{c}(\mathcal P)$ is completely isomorphic to $\mathcal L_{\alpha,k}^{r}(\mathcal P)$ with equivalent norms.
\end{corollary}

\section{The proof of Theorem \ref{b3}(ii)}
\hskip\parindent
In this section, we are going to prove Theorem \ref{b3}(ii). That is, under the condition $\Gamma^{2}\geq0$, the column little Campanato space $\ell_{\alpha}^{c}(\mathcal P)$ and the row little Campanato space $\ell_{\alpha}^{r}(\mathcal P)$ via the subordinated Poisson semigroup $(P_t)_{t\geq0}$ are isomorphic with equivalent norms for $0<\alpha<\frac{1}{2}$.

Let $(T_t)_{t\geq0}$ be a Markov semigroup and $(P_t)_{t\geq0}$ be its subordinated Poisson semigroup.
For each $0<\alpha<1$,  we define
\begin{align*}
 \ell_{\alpha}^{c}(\mathcal P)=\left\{f\in \mathcal{M}: \|f\|_{\ell_{\alpha}^{c}(\mathcal P)}<\infty\right\},
\end{align*}
where
\begin{align*}
 &\|f\|_{\ell_{\alpha}^{c}(\mathcal P)}=\|f\|_{\infty}+\sup_{t>0}\frac{1}{t^{\alpha}} \left\|P_{t}|f|^{2}-|P_{t}f|^{2}\right\|_{\infty}^{\frac{1}{2}}.
\end{align*}
It is well known that for a fixed $t>0$, the following triangle inequality holds for any $f,g\in\mathcal M$
\begin{align}\label{lc}
\left\|P_{t}|f+g|^{2}-\left|P_{t} (f+g)\right|^{2}\right\|^{\frac{1}{2}}_{\infty}\leq \left\|P_{t}|f|^{2}-\left|P_{t} f\right|^{2}\right\|^{\frac{1}{2}}_{\infty}+\left\|P_{t}|g|^{2}-\left|P_{t} g\right|^{2}\right\|^{\frac{1}{2}}_{\infty}
\end{align}
(see e.g. \cite[Proposition 2.1]{mj12}).
Then it is trivial that the quantity $\|\cdot\|_{\ell_{\alpha}^{c}(\mathcal P)}$ defines a norm.
 One defines the row space $\ell_{\alpha}^{r}(\mathcal P)$ with the norm $\|f\|_{\ell_{\alpha}^{r}(\mathcal P)}=\|f^{*}\|_{\ell_{\alpha}^{c}(\mathcal P)}$ and the mixture space $\ell_{\alpha}^{cr}(\mathcal P)$  with the norm $\|f\|_{\ell_{\alpha}^{cr}(\mathcal P)}=\max\{\|f\|_{\ell_{\alpha}^{c}(\mathcal P)},\|f\|_{\ell_{\alpha}^{r}(\mathcal P)}\}$.

The completeness of these spaces can be proved as in the case of $\mathcal L_{\alpha}^{c}(\mathcal P)$, and we omit the details.
\begin{proposition}\label{p9}
Let $0<\alpha<1$. The space $\ell_{\alpha}^{\dag}(\mathcal P)$ is complete with respect to the norm $\|\cdot\|_{\ell_{\alpha}^{\dag}(\mathcal P)}$, where $\dag=\{c,r,cr\}$.
\end{proposition}

\begin{remark}\label{lk}
We point out that the inequality \eqref{lc} holds true for any unital completely positive operator replacing $P_{t}$.
\end{remark}

Now we introduce the $\Gamma^{2}$ criterion.
\begin{definition}\label{04}
  P.A Meyer's gradient form $\Gamma$ (also called ``Carr$\acute{\mathrm{e}}$ du Champ") associated with $(T_{t})_{t\geq0}$ generated by $A$ is defined as
\begin{align}\label{06}
  2 \Gamma(f, g)=Af^{*} g+f^{*}Ag-A(f^{*} g)
\end{align}
for $f^{*}, g, f^{*} g\in \mathrm{dom}_{\infty}(A)$.
\end{definition}
When $f=g$, we simply write $\Gamma(f)=\Gamma
(f,f)$. We remark that $\Gamma$ coincides with the usual gradient form on $\mathbb R^d$ when $A =-\Delta$ since $\Gamma(f, g)=\nabla f^{*}\nabla g$.

Let $t>0$ and $f_{t}$ be a $\mathcal M$-valued function on a measure space $(\Omega,d\mu)$ such that $\Gamma(f_{t})$ is weakly measurable. By \cite[Lemma 3.1]{mj12}, we have the following Kadison-Schwarz inequality:
\begin{align}\label{pm}
\Gamma(\int_{\Omega}f_{t}d\mu(t))\leq\int_{\Omega}|d\mu|(t)\int_{\Omega}\Gamma(f_{t})|d\mu|(t).
\end{align}

Bakry-$\acute{\mathrm{E}}$mery's iterated gradient form $\Gamma^{2}$ associated with $(T_{t})_{t\geq0}$ is defined as
\begin{align}\label{07}
   2\Gamma^{2}(f, g)=\Gamma\left(Af, g\right)+\Gamma \left(f, Ag\right)-A\Gamma\left(f, g\right)
\end{align}
whenever $f,g$ are nice such that all the quantities on the right hand side make senses.


According to \cite[Chapter 3, proposition 3.3]{B06} and \cite[Lemma 1.2]{m00}, the semigroup $(T_{t})_{t\geq0}$ satisfies the $\Gamma^{2}\geq0$, that is, $\Gamma^2(f,f)\geq0$ for all nice $f$, if and only if
\begin{align*}
  T_r\Gamma(T_{s+t}f)\leq T_{r+t}\Gamma(T_{s}f)
\end{align*}
holds for all $r,s,t>0$ and $f\in\mathcal M$. Here the quantity $T_u\Gamma(T_{v}f)$ should be read as
\begin{align*}
 T_{u}((AT_{v}f^{\ast})(T_{v}f))+T_{u}((T_{v}f^{\ast})(AT_{v}f))-AT_{u}((T_{v}f^{\ast})(T_{v}f)),
\end{align*}
which obviously makes sense for $f\in\mathcal M$ since $T_vf\in \mathrm{dom}_\infty(A)$.
However, $\Gamma^{2}\geq0$ is not always true. For instance, when $A$ is a Laplace-Beltrami operator on a complete manifold, the $\Gamma^{2}\geq0$ criterion holds if and only if the manifold has nonnegative Ricci curvature everywhere, see \cite{BFG12}.

We will use the following lemma due to P.A. Meyer, see also \cite{mj12,mj10} for a noncommutative version. The Carr\'e du champ associated to the generator $A^\frac12$ is denoted by $\Gamma_{A^\frac12}$.
\begin{lemma}\label{e.1}
Let $(T_t)_{t\geq0}$ be a Markov semigroup and $(P_t)_{t\geq0}$ be its subordinated Poisson semigroup. Then the following statements hold.
 \begin{itemize}
   \item [\rm{(i)}] For any $f\in \mathcal{M}$, $s, t>0$,
   $$T_{s}|f|^{2}-\left|T_{s} f\right|^{2}=2 \int_{0}^{s} T_{s-t} \Gamma\left(T_{t} f\right) d t.$$
  \item [\rm{(ii)}] For any $f\in\mathcal{M}$, $s, t>0$, it holds that
  $$P_{t} \Gamma_{A^{\frac{1}{2}}}\left(P_{s} f\right)=\int_{0}^{\infty} P_{t+v} \Gamma\left(P_{s+v} f\right) d v+\int_{0}^{\infty} P_{t+v}\left|\frac{\partial P_{s+v} f}{\partial v}\right|^{2} d v.$$
 \end{itemize}
\end{lemma}
From the second assertion in the above lemma, it is easy to check that $\Gamma^{2}\geq0$ actually implies that $\Gamma_{A^{\frac{1}{2}}}^{2}\geq0$, that is,  for any $r,s,t>0$ and $f\in\mathcal M$,
\begin{align}\label{Gamma12}P_{r} \Gamma_{A^{\frac{1}{2}}}\left(P_{s+t} f\right))\leq  P_{r+t} \Gamma_{A^{\frac{1}{2}}}\left(P_{s} f\right).\end{align} Indeed, firstly under the condition $\Gamma^{2}\geq0$, one can easily deduce that $P_r\Gamma(P_{s+t}f)\leq P_{r+t}\Gamma(P_{s}f)$
for any $r,s,t>0$ and $f\in\mathcal M$ via using \eqref{2} and \eqref{pm}. Then applying Lemma \ref{e.1}(ii) and the Kadison-Schwarz inequality in \eqref{1}, we get
\begin{align*}
  P_{r} \Gamma_{A^{\frac{1}{2}}}\left(P_{s+t} f\right)=&\int_{0}^{\infty} P_{r+v} \Gamma\left(P_{s+t+v} f\right) d v+\int_{0}^{\infty} P_{r+v}\left|\frac{\partial P_{s+t+v} f}{\partial v}\right|^{2} d v
\\  \leq& \int_{0}^{\infty} P_{r+t+v} \Gamma\left(P_{s+v} f\right) d v+\int_{0}^{\infty} P_{r+t+v}\left|\frac{\partial P_{s+v} f}{\partial v}\right|^{2} d v
\\ =& P_{r+t} \Gamma_{A^{\frac{1}{2}}}\left(P_{s} f\right).
\end{align*}

The proof of the following proposition follows the same steps of \rm{\cite[Proposition 2.4]{mj12}}, we include it here for the sake of completeness.
\begin{proposition}\label{e12}
Let $(P_{t})_{t\geq0}$ be the Markov semigroup associated to $(T_{t})_{t\geq0}$ and $f\in \mathcal{M}$.
\begin{itemize}
  \item [\rm{(i)}] For $0<\alpha<1$, we have
  \begin{align*}
 \|f\|_{\mathcal L_{\alpha}^{c}(\mathcal P)} \leq (2^{\alpha}+1)\|f\|_{\ell_{\alpha}^{c}(\mathcal P)}+\sup _{t>0}\frac{1}{t^{\alpha}}\left\|P_{t} f-P_{2 t} f\right\|_{\infty};
\end{align*}
  \item [\rm{(ii)}] If in addition $\Gamma^{2}\geq0$, we get for $0<\alpha<\frac{1}{2}$,
  \begin{align*}
  \|f\|_{\ell_{\alpha}^{c}(\mathcal P)}+\sup _{t>0}\frac{1}{t^{\alpha}}\left\|P_{t} f-P_{2 t} f\right\|_{\infty}\leq\left(\frac{\sqrt{2}}{1-2^{\alpha-\frac{1}{2}}}+1\right)\|f\|_{\mathcal L_{\alpha}^{c}(\mathcal P)}.
  \end{align*}
\end{itemize}
\end{proposition}
\begin{proof}
 (\rm{i}). Let $0<\alpha<1$. We just need to show
\begin{align}\label{m1}
 \sup_{t>0} \frac{1}{t^{\alpha}} \left\|P_{t}|f-P_{t}f|^{2}\right\|_{\infty}^{\frac{1}{2}}
  \leq(2^{\alpha}+1)\sup_{t>0}\frac{1}{t^{\alpha}}\left\|P_{t}|f|^{2}-\left|P_{t} f\right|^{2}\right\|^{\frac{1}{2}}_{\infty}
  +\sup_{t>0}\frac{1}{t^{\alpha}}\left\|P_{t}f-P_{2t}f\right\|_{\infty}.
\end{align}
By the triangle inequality, one gets
\begin{align*}
  \frac{1}{t^{\alpha}} \left\|P_{t}|f-P_{t}f|^{2}\right\|_{\infty}^{\frac{1}{2}}
 \leq&\frac{1}{t^{\alpha}} \left\|P_{t}|f-P_{t}f|^{2}-|P_{t}(f-P_{t}f)|^{2}\right\|_{\infty}^{\frac{1}{2}}+\frac{1}{t^{\alpha}}\left\| P_{t}(f-P_{t}f)\right\|_{\infty}.
\end{align*}
Furthermore, applying the triangle inequality \eqref{lc} and the Kadison-Schwarz inequality \eqref{1}, we have
\begin{align*}
&\frac{1}{t^{\alpha}} \left\|P_{t}|f-P_{t}f|^{2}\right\|_{\infty}^{\frac{1}{2}}
\\ \leq&\frac{1}{t^{\alpha}}\left\|P_{t}|f|^{2}-\left|P_{t} f\right|^{2}\right\|^{\frac{1}{2}}_{\infty}+\frac{1}{t^{\alpha}} \left\|P_{t}|P_{t}f|^{2}-|P_{2t}f|^{2}\right\|_{\infty}^{\frac{1}{2}}
+\frac{1}{t^{\alpha}}\left\|P_{t}f-P_{2t}f\right\|_{\infty}
\\ \leq&
\frac{1}{t^{\alpha}}\left\|P_{t}|f|^{2}-\left|P_{t} f\right|^{2}\right\|^{\frac{1}{2}}_{\infty}+\frac{1}{t^{\alpha}} \left\|P_{2t}|f|^{2}-|P_{2t}f|^{2}\right\|_{\infty}^{\frac{1}{2}}
+\frac{1}{t^{\alpha}}\left\|P_{t}f-P_{2t}f\right\|_{\infty}
\\ \leq&(2^{\alpha}+1)\sup_{t>0}\frac{1}{t^{\alpha}}\left\|P_{t}|f|^{2}-\left|P_{t} f\right|^{2}\right\|^{\frac{1}{2}}_{\infty}+\sup_{t>0}\frac{1}{t^{\alpha}}\left\|P_{t}f-P_{2t}f\right\|_{\infty}.
\end{align*}
Taking the supremum over $t>0$ on the left-hand side yields the assertion \eqref{m1}.

$(\rm{ii})$ Using the triangle inequality \eqref{lc} and applying Lemma \ref{e.1} (i) to the subordinated Poisson semigroup, we have for $t>0$ fixed,
\begin{align*}
\frac{1}{t^{\alpha}}\left\|P_{t}|f|^{2}-\left|P_{t} f\right|^{2}\right\|^{\frac{1}{2}}_{\infty}
 \leq&\frac{1}{t^{\alpha}}\left\|P_{t}\left|f-P_{t} f\right|^{2}-\left|P_{t} f-P_{2 t} f\right|^{2}\right\|^{\frac{1}{2}}_{\infty}+\frac{1}{t^{\alpha}}\left\|P_{t}\left|P_{t} f\right|^{2}-\left|P_{t} (P_{t} f)\right|^{2}\right\|^{\frac{1}{2}}_{\infty}
\\
=&\frac{1}{t^{\alpha}}\left\|P_{t}\left|f-P_{t} f\right|^{2}-\left|P_{t} f-P_{2 t} f\right|^{2}\right\|^{\frac{1}{2}}_{\infty}+\frac{1}{t^{\alpha}}\left\|2\int_{0}^{t} P_{ t- s} \Gamma_{ A^{\frac{1}{2}}}\left(P_{t+ s} f\right) d s\right\|^{\frac{1}{2}}_{\infty}
\\=&:I+II.
\end{align*}
The first term is easy to handle,
\begin{align*}
I\leq \frac{1}{t^{\alpha}}\left\|P_{t}\left|f-P_{t} f\right|^{2}\right\|^{\frac{1}{2}}_{\infty}\leq\sup_{t>0}\frac{1}{t^{\alpha}}\left\|P_{t}\left|f-P_{t} f\right|^{2}\right\|^{\frac{1}{2}}_{\infty}.
\end{align*}
For the second term, by \eqref{Gamma12}, one gets
\begin{align*}
II\leq& \frac{1}{t^{\alpha}}\left\|2\int_{0}^{t} P_{2 t-2 s} \Gamma_{ A^{\frac{1}{2}}}\left(P_{2 s} f\right) d s\right\|^{\frac{1}{2}}_{\infty}
\\ (v=2s)=&  \frac{1}{t^{\alpha}}\left\|\int_{0}^{2 t} P_{2 t-v} \Gamma_{ A^{\frac{1}{2}}}\left(P_{v} f\right) d v\right\|^{\frac{1}{2}}_{\infty}
\\ =&2^{\alpha-\frac{1}{2}}\frac{1}{(2t)^{\alpha}}\left\|P_{2t}|f|^{2}-\left|P_{2t} f\right|^{2}\right\|^{\frac{1}{2}}_{\infty}
\\ \leq& 2^{\alpha-\frac{1}{2}}\sup_{t>0}\frac{1}{t^{\alpha}}\left\|P_{t}|f|^{2}-\left|P_{t} f\right|^{2}\right\|^{\frac{1}{2}}_{\infty}.
\end{align*}
Taking the supremum over $t>0$, we thus get
\begin{align*}
  \sup_{t>0}\frac{1}{t^{\alpha}}\left\|P_{t}|f|^{2}-\left|P_{t} f\right|^{2}\right\|^{\frac{1}{2}}_{\infty}\leq &\sup_{t>0}\frac{\sqrt{2}}{t^{\alpha}}\left\|P_{t}\left|f-P_{t} f\right|^{2}\right\|^{\frac{1}{2}}_{\infty}+2^{\alpha-\frac{1}{2}}\sup_{t>0}\frac{1}{t^{\alpha}}\left\|P_{t}|f|^{2}-\left|P_{t} f\right|^{2}\right\|^{\frac{1}{2}}_{\infty}.
\end{align*}
Since $2^{\alpha-\frac{1}{2}}<1$ for $0<\alpha<\frac{1}{2}$, we then have
\begin{align*}
 \sup_{t>0}\frac{1}{t^{\alpha}}\left\|P_{t}|f|^{2}-\left|P_{t} f\right|^{2}\right\|^{\frac{1}{2}}_{\infty}\leq \frac{\sqrt{2}}{1-2^{\alpha-\frac{1}{2}}} \sup_{t>0}\frac{1}{t^{\alpha}}\left\|P_{t}\left|f-P_{t} f\right|^{2}\right\|^{\frac{1}{2}}_{\infty}.
\end{align*}
Thus, we obtain
\begin{align*}
   \|f\|_{\ell_{\alpha}^{c}(\mathcal P)}\leq\frac{\sqrt{2}}{1-2^{\alpha-\frac{1}{2}}}\sup_{t>0}\frac{1}{t^{\alpha}}\left\|P_{t}\left|f-P_{t} f\right|^{2}\right\|^{\frac{1}{2}}_{\infty}+\|f\|_{\infty}\leq\frac{\sqrt{2}}{1-2^{\alpha-\frac{1}{2}}}\|f\|_{\mathcal L_{\alpha}^{c}(\mathcal P)}.
\end{align*}
On the other hand, for any $t>0$,
\begin{align*}
 \frac{1}{t^{\alpha}}\left\|P_{t}(f-P_{t}f)\right\|_{\infty}\leq\frac{1}{t^{\alpha}}\left\|P_{t}|f-P_{t}f|^{2}\right\|^{\frac{1}{2}}_{\infty}\leq\|f\|_{\mathcal L_{\alpha}^{c}(\mathcal P)}.
 \end{align*}
\end{proof}
We will need to introduce the following quantity:
\begin{align*}
&\left\|f\right\|_{\mathcal L_{\alpha}^{c}(\partial)}=\|f\|_{\infty}+\sup_{t>0}\frac{1}{t^{\alpha}} \left\|P_{t}\int^{t}_{0}\left|\frac{\partial P_{s}f}{\partial s}\right|^{2}sds\right\|^{\frac{1}{2}}_{\infty}.
\end{align*}
In the case of $\alpha=0$, the quantity $\sup_{t>0}\left\|P_{t}\int^{t}_{0}|\frac{\partial P_{s}f}{\partial s}|^{2}sds\right\|^{\frac{1}{2}}_{\infty}$ is a Carleson measure BMO norm associated to semigroup introduced in \rm{\cite[eg. (2.4)]{mj12}}. Therefore, we obtain that, for any $\alpha>0$, $\left\|\cdot\right\|_{\mathcal L_{\alpha}^{c}(\partial)}$ are a norms on $\mathcal M$.
\begin{lemma}\label{e14}
Let $(T_{t})_{t\geq0}$ be a Markov semigroup, $(P_{t})_{t\geq0}$ its associated Poisson semigroup and $f\in \mathcal{M}$. Then, for $0<\alpha<1$,
 $$\sup_{t>0} \frac{1}{t^{\alpha}}\left\|P_{t} f-P_{2 t} f\right\|_{\infty}\lesssim_{\alpha}\left\|f\right\|_{\mathcal L_{\alpha}^{c}(\partial)} \lesssim\|f\|_{\ell_{\alpha}^{c}(\mathcal P)}.$$
\end{lemma}
 \begin{proof}
 To prove the first inequality, for one $t>0$ fixed, we use the Kadison-Schwarz inequality \eqref{1} to get
 \begin{align*}
  \frac{1}{t^{\alpha}}\left\|P_{t} f-P_{2 t} f\right\|_{\infty}=&\frac{1}{t^{\alpha}}\left\|P_{t} f-P_{\frac{3t}{2}}f+P_{\frac{3t}{2}}f-P_{2 t} f\right\|_{\infty}
  \\ \nonumber\leq&
 \frac{1}{t^{\alpha}} \left\|P_{t} f-P_{\frac{3t}{2}}f \right\|_{\infty}+\frac{1}{t^{\alpha}}\left\|P_{\frac{3t}{2}}f-P_{2 t} f\right\|_{\infty}
\\ \nonumber\leq&
 \frac{1}{t^{\alpha}} \left\|P_{\frac{3t}{4}}|P_{\frac{t}{4}} f-P_{\frac{3t}{4}}f |^{2} \right\|^{\frac{1}{2}}_{\infty}+\frac{1}{t^{\alpha}}\left\|P_{ t}|P_{\frac{t}{2}}f-P_{t} f|^{2}\right\|^{\frac{1}{2}}_{\infty}
 \\ \nonumber=&
 \frac{1}{t^{\alpha}} \left\|P_{\frac{3t}{4}}\left|\int^{\frac{3t}{4}}_{\frac{t}{4}}\frac{\partial P_{s}f}{\partial s} ds\right|^{2} \right\|^{\frac{1}{2}}_{\infty}+\frac{1}{t^{\alpha}}\left\|P_{ t}\left|\int^{t}_{\frac{t}{2}}\frac{\partial P_{s}f}{\partial s} ds\right|^{2}\right\|^{\frac{1}{2}}_{\infty}.
\end{align*}
Using the Kadison-Schwarz inequality \eqref{1.1}, we have
\begin{align*}
\frac{1}{t^{\alpha}}\left\|P_{t} f-P_{2 t} f\right\|_{\infty}\leq&
 \frac{1}{t^{\alpha}} \left\|P_{\frac{3t}{4}}
\int^{\frac{3t}{4}}_{\frac{t}{4}} \frac{t}{2}\left|\frac{\partial P_{s}f}{\partial s} \right|^{2} ds\right\|^{\frac{1}{2}}_{\infty}+\frac{1}{t^{\alpha}}\left\|P_{ t}\int^{t}_{\frac{t}{2}} \frac{t}{2}\left|\frac{\partial P_{s}f}{\partial s}\right|^{2} ds\right\|^{\frac{1}{2}}_{\infty}
\\ \nonumber\leq&
 \frac{1}{t^{\alpha}} \left\|P_{\frac{3t}{4}}
\int^{\frac{3t}{4}}_{\frac{t}{4}}2\left|\frac{\partial P_{s}f}{\partial s} \right|^{2} sds\right\|^{\frac{1}{2}}_{\infty}+\frac{1}{t^{\alpha}}\left\|P_{ t}\int^{t}_{\frac{t}{2}}\left|\frac{\partial P_{s}f}{\partial s}\right|^{2}s ds\right\|^{\frac{1}{2}}_{\infty}
\\ \nonumber\leq&
 (\frac{3}{4})^{\alpha}\sqrt{2}\frac{1}{(\frac{3t}{4})^{\alpha}} \left\|P_{\frac{3t}{4}}
\int^{\frac{3t}{4}}_{0}\left|\frac{\partial P_{s}f}{\partial s} \right|^{2} sds\right\|^{\frac{1}{2}}_{\infty}+\frac{1}{t^{\alpha}}\left\|P_{ t}\int^{t}_{0}\left|\frac{\partial P_{s}f}{\partial s}\right|^{2}s ds\right\|^{\frac{1}{2}}_{\infty}
\\ \nonumber\leq&(2^{\frac{1}{2}-2\alpha}3^{\alpha}+1)\left\|f\right\|_{\mathcal L_{\alpha}^{c}(\partial)}.
\end{align*}
Taking the supremum over $t>0$ on the left-hand side, we get the desired result.

We now prove the second inequality. Fix $s,t>0$.
Let $$P_{t}\hat{\Gamma}\left(P_{s} f\right)=P_{t}\Gamma\left(P_{s} f\right)+P_{t}\left|\frac{\partial P_{s}f}{\partial s}\right|^{2}.$$
Applying the Kadison-Schwarz inequality \eqref{1}, we deduce
\begin{align*}
 \frac{1}{t^{\alpha}} \left\|P_{t}\int^{t}_{0}\left|\frac{\partial P_{s}f}{\partial s}\right|^{2}sds\right\|^{\frac{1}{2}}_{\infty}\leq&\frac{1}{t^{\alpha}} \left\|\int^{t}_{0}P_{t+\frac{s}{2}}\left|\frac{\partial P_{\frac{s}{2}}f}{\partial s}\right|^{2}sds\right\|^{\frac{1}{2}}_{\infty}
  \\ \nonumber(\frac{s}{2}=v)=&
  \frac{2}{t^{\alpha}} \left\|\int^{\frac{t}{2}}_{0}P_{t+v}\left|\frac{\partial P_{v}f}{\partial v}\right|^{2}vdv\right\|^{\frac{1}{2}}_{\infty}
\\ \nonumber\leq&
  \frac{2}{t^{\alpha}} \left\|\int^{t}_{0}P_{t+v}\left|\frac{\partial P_{v}f}{\partial v}\right|^{2}vdv\right\|^{\frac{1}{2}}_{\infty}\\
  \leq&
  \frac{2}{t^{\alpha}} \left\|2\int^{t}_{0}\frac{(t-v)v+v^{2}}{t+v}P_{t+v}\hat{\Gamma}\left( P_{v} f\right)dv\right\|^{\frac{1}{2}}_{\infty},
  \end{align*}
  which is trivially smaller than
  \begin{align*}
  &\frac{2}{t^{\alpha}} \left\|2\int^{\infty}_{0}\frac{(t-v)v+v^{2}-(t-v)\max\{0,v-t\}-\max\{0,v-t\}^{2}}{t+v}P_{t+v}\hat{\Gamma}\left( P_{v} f\right)dv\right\|^{\frac{1}{2}}_{\infty}
\\ \nonumber=&
  \frac{2}{t^{\alpha}} \left\|2\int^{\infty}_{0}\int^{v}_{\max\{0,v-t\}}\frac{t-v+2u}{t+v}duP_{t+v}\hat{\Gamma}\left( P_{v} f\right)dv\right\|^{\frac{1}{2}}_{\infty}
 \\ \nonumber\leq&
  \frac{2}{t^{\alpha}} \left\|2\int^{\infty}_{0}\int^{v}_{\max\{0,v-t\}}P_{t-v+2u}\hat{\Gamma}\left( P_{v} f\right)dudv\right\|^{\frac{1}{2}}_{\infty}.
\end{align*}
In the last inequality, we have used the monotonicity \eqref{4} of Poisson semigroup.
Then, by twice the Fubini theorem and using the formula $\Gamma_{A^{\frac{1}{2}}}\left(P_{s} f\right)$ appeared in Lemma \ref{e.1}(ii), we obtain that
\begin{align*}
 \frac{1}{t^{\alpha}} \left\|P_{t}\int^{t}_{0}\left|\frac{\partial P_{s}f}{\partial s}\right|^{2}sds\right\|^{\frac{1}{2}}_{\infty}
 \leq&
  \frac{2}{t^{\alpha}} \left\|2\int^{\infty}_{0}\int^{u+t}_{u}P_{t-v+2u}\hat{\Gamma}\left( P_{v} f\right)dvdu\right\|^{\frac{1}{2}}_{\infty}
  \\ \nonumber(v=s+u)=&
  \frac{2}{t^{\alpha}} \left\|2\int^{t}_{0}\int^{\infty}_{0}P_{t-s+u}\hat{\Gamma}\left(P_{s+u} f\right)duds\right\|^{\frac{1}{2}}_{\infty}
 \\ \nonumber
=&
  \frac{2}{t^{\alpha}} \left\|2\int^{t}_{0}P_{t-s}\Gamma_{A^{\frac{1}{2}}}\left(P_{s} f\right)ds\right\|^{\frac{1}{2}}_{\infty}
\\ \nonumber=&
  \frac{2}{t^{\alpha}} \left\|P_{t}|f|^{2}-|P_{t}f|^{2}\right\|^{\frac{1}{2}}_{\infty}.
\end{align*}
This implies
\begin{align*}
\left\|f\right\|_{\mathcal L_{\alpha}^{c}(\partial
)}
\leq&
  2 \left\|f\right\|_{\ell_{\alpha}^{c}(\mathcal{P})}.
\end{align*}
\end{proof}
By Proposition \ref{e12} and Lemma \ref{e14}, the following result comes out naturally.

 \begin{theorem}\label{e15}
  Let $(T_{t})_{t\geq0}$ be a Markov semigroup and $(P_{t})_{t\geq0}$ be the associated Poisson semigroup satisfying $\Gamma^{2}\geq0$. Then, for any $0<\alpha<\frac{1}{2}$, the norms $\|f\|_{\mathcal L_{\alpha}^{\dag}(\mathcal P)}$ and $\left\|f\right\|_{\ell_{\alpha}^{\dag}(\mathcal{P})}$ are equivalent on $\mathcal{M}$, where $\dag=\{c,r,cr\}$.
 \end{theorem}
Together with Theorem \ref{e15} and Theorem \ref{b3}(i), we obtain Theorem \ref{b3}(ii).

\section{$\mathcal{L}_{\alpha}^{c}$ associated with general semigroups}
\hskip\parindent
In this section, we introduce the Campanato/Lipschitz spaces associated to the semigroup $\mathcal T=(T_{t})_{t\geq0}$ itself and explore their connection with the ones associated to its subordinated Poisson semigroup. 

Let $(T_{t})_{t\geq0}$ be a Markov semigroup acting on a finite von Neumann algebra $\mathcal{M}$ and $(P_{t})_{t\geq0}$ be its subordinated semigroup. For $0<\alpha<1$ and $f\in \mathcal{M}$, the spaces $\mathcal{L}_{\alpha}^{c}(\mathcal T)$, $\ell_{\alpha}^{c}(\mathcal P)$ and $\Lambda_{\alpha}(\mathcal{T})$ can be defined as
\begin{align*}
   &\mathcal{L}_{\alpha}^{c}(\mathcal T)=\left\{f\in\mathcal{M}: \,\, \left\|f\right\|_{\mathcal{L}_{\alpha}^{c}(\mathcal T)}<\infty\right\},
  \\&    \ell_{\alpha}^{c}(\mathcal T)=\left\{f\in \mathcal{M}: \|f\|_{\ell_{\alpha}^{c}(\mathcal T)}<\infty\right\},
  \\& \Lambda_{\alpha}(\mathcal{T})=\left\{f\in\mathcal{M}:\,\,\,\|f\|_{\Lambda_{\alpha}(\mathcal{T})}<\infty\right\},
 \end{align*}
 where
 \begin{align*}
 &\left\|f\right\|_{\mathcal{L}_{\alpha}^{c}(\mathcal T)}=\|f\|_{\infty}+\sup_{t>0}\frac{1}{t^{\alpha}} \left\|T_{t}|f-T_{t}f|^{2}\right\|_{\infty}^{\frac{1}{2}}.
\\&\left\|f\right\|_{\ell_{\alpha}^{c}(\mathcal T)}=\|f\|_{\infty}+\sup_{t>0}\frac{1}{t^{\alpha}} \left\|T_{t}|f|^{2}-|T_{t}f|^{2}\right\|_{\infty}^{\frac{1}{2}}.
\\&   \left\|f\right\|_{\Lambda_{\alpha}({\mathcal T})}=\|f\|_{\infty}+\sup_{t>0}\frac{1}{t^{\alpha-1}} \left\|\frac{\partial T_{t}f}{\partial t}\right\|_{\infty}.
\end{align*}
Similar to $\left\|\cdot\right\|_{\mathcal{L}_{\alpha}^{c}(\mathcal P)}$, $\left\|\cdot\right\|_{\Lambda_{\alpha}({\mathcal P})}$ and $\left\|\cdot\right\|_{\ell_{\alpha}^{c}(\mathcal P)}$, we know that $ \left\|\cdot\right\|_{\mathcal{L}_{\alpha}^{c}(\mathcal T)}$, $\left\|\cdot\right\|_{\Lambda_{\alpha}({\mathcal T})}$ and $ \left\|\cdot\right\|_{\ell_{\alpha}^{c}(\mathcal T)}$ are norms on $\mathcal M$, too.

As before, one defines the row norms and the resulting row/mixture spaces  $\mathcal{L}_{\alpha}^{r}(\mathcal T)$, $\ell_{\alpha}^{r}(\mathcal T)$, $\mathcal{L}_{\alpha}^{cr}(\mathcal T)$ and $\ell_{\alpha}^{cr}(\mathcal T)$.

\begin{proposition}
Let $0<\alpha<1$. The spaces $\mathcal{L}_{\alpha}^{\dag}(\mathcal T)$ and $ \ell_{\alpha}^{\dag}(\mathcal T)$ are Banach spaces, where $(\dag=c,r,cr)$.
\end{proposition}
\begin{lemma}\label{q1}
Suppose $(T_{t})_{t\geq0}$ is quasi-monotone. Then, for $0<\alpha<1$,
\begin{align*}
\|f\|_{\mathcal L_{\alpha}^{c}(\mathcal P)} \simeq \|f\|_{\infty}+\sup_{t}\frac{1}{t^{\alpha}}\left\|T_{t^{2}}|f-P_{t}f|^{2}\right\|_{\infty}^{\frac{1}{2}}.
\end{align*}
\end{lemma}
 \begin{proof}
If $(T_{t})_{t\geq0}$ is quasi-increasing, by Definition \ref{p0}, there exists a constant $\beta\geq0$ such that for $t^2<u<2t^2$,
\begin{align*}
  \frac{T_{t^{2}}}{t^{2\beta}}\leq \frac{T_{u}}{u^{\beta}}\leq  \frac{T_{2t^{2}}}{(2t)^{2\beta}}.
\end{align*}
We then have for a fixed $t>0$,
 \begin{align*}
  P_{t}=&\frac{1}{2 \sqrt{\pi}} \int_{0}^{\infty} t e^{-\frac{t^{2}}{4 u}} u^{-\frac{3}{2}} T_{u} d u
 \\ \geqslant& \frac{1}{2 \sqrt{\pi}} \int_{t^{2}}^{2 t^{2}} t e^{-\frac{t^{2}}{4 u}} u^{-\frac{3}{2}} T_{u} d u
\\  \geqslant& \frac{1}{2 \sqrt{\pi}} \int_{t^{2}}^{2 t^{2}} t e^{-\frac{t^{2}}{4 u}} u^{-\frac{3}{2}}(\frac{u}{t^{2}})^{\beta}T_{t^{2}} d u
 \\  \geqslant& T_{t^{2}}\left\{\frac{1}{2 \sqrt{\pi}} \int_{t^{2}}^{2 t^{2}} t e^{-\frac{t^{2}}{4 u}} u^{-\frac{3}{2}} d u\right\}
 \\(v=\frac{t^{2}}{4 u})=&T_{t^{2}}\left\{\frac{1}{\sqrt{\pi}} \int_{\frac{1}{8}}^{\frac{1}{4}}  e^{-v}v^{-\frac{1}{2}} d v\right\}.
 \end{align*}
 Similarly, if $(T_{t})_{t\geq0}$ is quasi-decreasing, i.e. there exists a constant $\beta\geq0$ such that
 $\frac{T_{t^{2}}}{t^{2\beta}}\leq \frac{T_{u}}{u^{\beta}}\leq  \frac{T_{\frac{t^{2}}{2}}}{(\frac{t^{2}}{2})^{2\beta}},$ one then can easily verify that $P_{t}\gtrsim T_{t^{2}},$ too.
Thus we can derive from the above two inequalities that $T_{t^{2}}(\varphi) \lesssim P_{t}(\varphi)$  for any positive $\varphi\in\mathcal M$. This tells that, for $0<\alpha<1$,
\begin{align*}
\sup_{t>0}\frac{1}{t^{\alpha}}\left\|T_{t^{2}}|f-P_{t}f|^{2}\right\|_{\infty}^{\frac{1}{2}}\lesssim\sup_{t>0}\frac{1}{t^{\alpha}}\left\|P_{t}|f-P_{t}f|^{2}\right\|_{\infty}^{\frac{1}{2}}.
\end{align*}
On the other hand, for $t>0$ fixed, we have
\begin{align*}
\frac{1}{t^{\alpha}}\left\|P_{t}|f-P_{t}f|^{2}\right\|_{\infty}^{\frac{1}{2}}=&\frac{1}{t^{\alpha}}\left\|\frac{1}{2\sqrt{\pi}}\int^{\infty}_{0}te^{-\frac{t^{2}}{4u}}u^{-\frac{3}{2}}T_{u} |f-P_{t}f|^{2}du\right\|_{\infty}^{\frac{1}{2}}
\\ \leq&\frac{1}{t^{\alpha}}\left(\frac{1}{2\sqrt{\pi}}\int^{t^{2}}_{0}te^{-\frac{t^{2}}{4u}}u^{-\frac{3}{2}}\left\|T_{u} |f-P_{t}f|^{2}\right\|_{\infty}du\right)^{\frac{1}{2}}
\\ &+\frac{1}{t^{\alpha}}\left(\frac{1}{2\sqrt{\pi}}\int^{\infty}_{t^{2}}te^{-\frac{t^{2}}{4u}}u^{-\frac{3}{2}}\left\|T_{u} |f-P_{t}f|^{2}\right\|_{\infty}du\right)^{\frac{1}{2}}.
\end{align*}
For $u\geq t^{2}$,
 \begin{align*}
 \left\|T_{u} |f-P_{t}f|^{2}\right\|_{\infty}=\left\|T_{u-t^{2}} T_{t^{2}}|f-P_{t}f|^{2}\right\|_{\infty}\leq\left\| T_{t^{2}}|f-P_{t}f|^{2}\right\|_{\infty}.
 \end{align*}
For $u<t^{2}$, let $n$ be the biggest integer not bigger than $\frac{t}{\sqrt{u}}$. We have
\begin{align*}
&\left\|T_{u}\left|f-P_{t}f\right|^{2}\right\|_{\infty}^{\frac{1}{2}} \\
\leq &\left\|T_{u}\left|f-P_{\sqrt{u}} f\right|^{2}\right\|_{\infty}^{\frac{1}{2}}+\cdots
+\left\|T_{u}\left|P_{(n-1) \sqrt{u}} f-P_{n \sqrt{u}}f\right|^{2}\right\|_{\infty}^{\frac{1}{2}}+\left\|T_{u}\left|P_{n \sqrt{u}} f-P_{t} f\right|^{2}\right\|_{\infty}^{\frac{1}{2}} \\
\leq&\left\|T_{u}\left|f-P_{\sqrt{u}} f\right|^{2}\right\|_{\infty}^{\frac{1}{2}}+\cdots+\left\|P_{(n-1) \sqrt{u}} T_{u}\left|f-P_{\sqrt{u}}f\right|^{2}\right\|_{\infty}^{\frac{1}{2}}
\\&+\left\|P_{n \sqrt{u}} T_{u-(t-n \sqrt{u})^{2}} T_{(t-n \sqrt{u})^{2}}\left|f-P_{t-n \sqrt{u}}f\right|^{2}\right\|_{\infty}^{\frac{1}{2}}\\
\leq&\frac{t}{\sqrt{u}} \frac{t^{\alpha}}{\sqrt{u}^{\alpha}}\left\|T_{u}\left|f-P_{\sqrt{u}} f\right|^{2}\right\|_{\infty}^{\frac{1}{2}}+\frac{ t^{\alpha}}{(t-n \sqrt{u})^{\alpha}}\left\|T_{(t-n \sqrt{u})^{2}}\left|f-P_{t-n \sqrt{u}}f\right|^{2}\right\|_{\infty}^{\frac{1}{2}}\\
\leq&  t^{\alpha}(1+\frac{t}{\sqrt{u}})\sup_{t}\frac{1}{t^{\alpha}}\left\|T_{t^{2}}\left|f-P_{t} f\right|^{2}\right\|_{\infty}^{\frac{1}{2}}
\\ \leq& 2\frac{t^{\alpha+1}}{\sqrt{u}}\sup_{t}\frac{1}{t^{\alpha}}\left\|T_{t^{2}}\left|f-P_{t} f\right|^{2}\right\|_{\infty}^{\frac{1}{2}}.
\end{align*}
Therefore,
\begin{align*}
 \frac{1}{t^{\alpha}}\left\|P_{t}|f-P_{t}f|^{2}\right\|_{\infty}^{\frac{1}{2}} \leq&\left(\frac{1}{2\sqrt{\pi}}\int^{t^{2}}_{0}te^{-\frac{t^{2}}{4u}}u^{-\frac{3}{2}}\frac{t^{2}}{u}du\right)^{\frac{1}{2}}\sup_{t}\frac{1}{t^{\alpha}}\left\|T_{t^{2}}\left|f-P_{t} f\right|^{2}\right\|_{\infty}^{\frac{1}{2}}
\\ &+\left(\frac{1}{2\sqrt{\pi}}\int^{\infty}_{t^{2}}te^{-\frac{t^{2}}{4u}}u^{-\frac{3}{2}}du\right)^{\frac{1}{2}}\sup_{t}\frac{1}{t^{\alpha}}\left\|T_{t^{2}}\left|f-P_{t} f\right|^{2}\right\|_{\infty}^{\frac{1}{2}}
\\ \lesssim & \sup_{t}\frac{1}{t^{\alpha}}\left\|T_{t^{2}}\left|f-P_{t} f\right|^{2}\right\|_{\infty}^{\frac{1}{2}}.
\end{align*}
Taking the supremum over $t>0$ on the left-hand side yields the assertion.
\end{proof}
We now come to prove our main results in this section.
\begin{proof}[Proof of Theorem \ref{q2}]
To show \rm{(i)}:
for any $0<\alpha<1$,
$$\|f\|_{\mathcal L_{\alpha}^{c}(\mathcal P)}\lesssim \|f\|_{\mathcal L_{\frac{\alpha}{2}}^{c}(\mathcal T)}.$$
By Lemma $\ref{q1}$, it suffices to show
\begin{align*}
  \|f\|_{\infty}+\sup_{t>0}\frac{1}{t^{\alpha}}\left\|T_{t^{2}}|f-P_{t}f|^{2}\right\|_{\infty}^{\frac{1}{2}}\lesssim \|f\|_{\mathcal L_{\frac{\alpha}{2}}^{c}(\mathcal T)}.
 \end{align*}
Fix $t>0$, we apply $(\ref{2})$ and obtain
\begin{align*}
 \frac{1}{t^{\alpha}}\left\|T_{t^{2}}|f-P_{t}f|^{2}\right\|_{\infty}^{\frac{1}{2}}=&\frac{1}{t^{\alpha}}\left\|T_{t^{2}}\left|\frac{1}{2\sqrt{\pi}}\int^{\infty}_{0}te^{-\frac{t^{2}}{4u}}u^{-\frac{3}{2}}(f-T_{u} f) du\right|^{2}\right\|_{\infty}^{\frac{1}{2}}
 \\ \nonumber\leq&
 \frac{1}{t^{\alpha}}\left\|\frac{1}{2\sqrt{\pi}}\int^{\infty}_{0}te^{-\frac{t^{2}}{4u}}u^{-\frac{3}{2}}T_{t^{2}}\left|f-T_{u} f \right|^{2}du\right\|_{\infty}^{\frac{1}{2}}
\\ \nonumber\leq&
 \left(\frac{1}{2\sqrt{\pi}}\int^{\infty}_{0}te^{-\frac{t^{2}}{4u}}u^{-\frac{3}{2}}\frac{1}{t^{2\alpha}}\left\|T_{t^{2}}\left|f-T_{u} f \right|^{2}\right\|_{\infty}du\right)^{\frac{1}{2}}.
\end{align*}
For $u\leq t^{2}$,
\begin{align*}
 \left\|T_{t^{2}}\left|f-T_{u} f \right|^{2}\right\|_{\infty}
 \leq&\left\|T_{u}\left|f-T_{u}f  \right|^{2}\right\|_{\infty}
\\ \leq&\frac{t^{2\alpha}}{u^{\alpha}}\left\|T_{u}\left|f-T_{u} f \right|^{2}\right\|_{\infty}
\\ \leq& t^{2\alpha}\left\{\sup_{t>0}\frac{1}{t^{\frac{\alpha}{2}}}\left\|T_{t}\left|f-T_{t} f \right|^{2}\right\|_{\infty}^{\frac{1}{2}}\right\}^{2}.
\end{align*}
For $u> t^{2}$, denoting by $m$ the biggest integer not bigger than $\log_{2}\frac{u}{t^{2}}$, by the Kadison-Schwarz inequality \eqref{1}, we have
\begin{align*}
&\frac{1}{t^{2\alpha}}\left\|T_{t^{2}}\left|f-T_{u} f \right|^{2}\right\|_{\infty}
\\ \leq &\frac{1}{t^{2\alpha}}\left\|T_{t^{2}}\left|f-T_{t^{2}} f \right|^{2}\right\|_{\infty}+\frac{1}{t^{2\alpha}}\left\|T_{t^{2}}\left|T_{t^{2}} f-T_{2t^{2}} f \right|^{2}\right\|_{\infty}
+\frac{1}{t^{2\alpha}}\left\|T_{t^{2}}\left|T_{2t^{2}} f-T_{4t^{2}} f \right|^{2}\right\|_{\infty}
\\&+
\cdots +\frac{1}{t^{2\alpha}}\left\|T_{t^{2}}\left|T_{2^{m}t^{2}}f-T_{u} f \right|^{2}\right\|_{\infty}
\\ \leq&
\frac{1}{t^{2\alpha}}\left\|T_{t^{2}}\left|f-T_{t^{2}} f \right|^{2}\right\|_{\infty}+\frac{1}{t^{2\alpha}}\left\|T_{t^{2}}\left| f-T_{t^{2}} f \right|^{2}\right\|_{\infty}+\frac{2^{\alpha}}{(2t^{2})^{\alpha}}\left\|T_{2t^{2}}\left| f-T_{2t^{2}} f \right|^{2}\right\|_{\infty}
\\ &+\cdots + \frac{(u-2^{m}t^{2})^{\alpha}}{t^{2\alpha}}\frac{1}{(u-2^{m}t^{2})^{\alpha}}\left\|T_{u-2^{m}t^{2}}\left|f-T_{u-2^{m}t^{2}} f \right|^{2}\right\|_{\infty}
\\ \lesssim&\frac{1}{t^{2\alpha}}(\log_{2}\frac{u}{t^{2}}+2)\left\{\sup_{t>0}\frac{1}{t^{\frac{\alpha}{2}}}\left\|T_{t}\left|f-T_{t} f \right|^{2}\right\|_{\infty}^{\frac{1}{2}}\right\}^{2}.
\end{align*}
Hence, we get
 \begin{align*}
   \frac{1}{t^{\alpha}}\left\|T_{t^{2}}|f-P_{t}f|^{2}\right\|_{\infty}^{\frac{1}{2}}
   \leq&
 \left(\frac{1}{2\sqrt{\pi}}\int^{t^{2}}_{0}te^{-\frac{t^{2}}{4u}}u^{-\frac{3}{2}}du\right)^{\frac{1}{2}}\left\{\sup_{t>0}\frac{1}{t^{\frac{\alpha}{2}}}\left\|T_{t}\left|f-T_{t} f \right|^{2}\right\|_{\infty}^{\frac{1}{2}}\right\}
 \\ &+\left(\frac{1}{2\sqrt{\pi}}\int^{\infty}_{t^{2}}(\log_{2}\frac{u}{t^{2}}+2)te^{-\frac{t^{2}}{4u}}u^{-\frac{3}{2}}du\right)^{\frac{1}{2}}\left\{\sup_{t>0}\frac{1}{t^{\frac{\alpha}{2}}}\left\|T_{t}\left|f-T_{t} f \right|^{2}\right\|_{\infty}^{\frac{1}{2}}\right\}
\\  \lesssim &
\sup_{t>0}\frac{1}{t^{\frac{\alpha}{2}}}\left\|T_{t}\left|f-T_{t} f \right|^{2}\right\|_{\infty}^{\frac{1}{2}},
\end{align*}
which implies \rm{(i)}.

 Using the same tricks in Theorem \ref{e25}, one gets
 \begin{align*}
   \|f\|_{\mathcal L_{\alpha}^{c}(\mathcal T)}\lesssim_{\alpha}\|f\|_{\Lambda_{\alpha}(\mathcal T)}
 \end{align*}
for any $\alpha>0$. Then using the results of Theorem \ref{e25} and \rm{(i)}, \rm{(ii)} can be shown easily by the fact that $\|f\|_{\Lambda_{\alpha}(\mathcal P)}\simeq_{\alpha}\|f\|_{\mathcal L_{\alpha}^{c}(\mathcal P)}\lesssim \|f\|_{\mathcal L_{\frac{\alpha}{2}}^{c}(\mathcal T)}\lesssim_{\alpha}\|f\|_{\Lambda_{\frac{\alpha}{2}}(\mathcal T)} $ for any $0<\alpha<1$.

We now deal with \rm{(iii)}: If $(T_{t})_{t\geq0}$ has the $\Gamma^{2}\geq0$ property and $0<\alpha<1$, $$\|f\|_{\ell_{\frac{\alpha}{2}}^{c}(\mathcal T)}\lesssim \|f\|_{\ell_{\alpha}^{c}(\mathcal P)}.$$
By Remark \ref{lk}, Lemma \ref{q1} and Lemma \ref{e.1}(i), for $t>0$ fixed, we deduce
\begin{align}\label{kas}
& \frac{1}{t^{\frac{\alpha}{2}}}\left\|T_{t}|f|^{2}-|T_{t}f|^{2}\right\|^{\frac{1}{2}}_{\infty}
\\ \nonumber \leq&\frac{1}{t^{\frac{\alpha}{2}}}\left\|T_{t}|f-P_{\sqrt{t}}f|^{2}- |T_{t}(f-P_{\sqrt{t}}f)|^{2}\right\|^{\frac{1}{2}}_{\infty}+\frac{1}{t^{\frac{\alpha}{2}}}\left\|T_{t}|P_{\sqrt{t}}f|^{2}- |T_{t}(P_{\sqrt{t}}f)|^{2}\right\|^{\frac{1}{2}}_{\infty}
\\ \nonumber\leq&\frac{1}{t^{\frac{\alpha}{2}}}\left\|T_{t}|f-P_{\sqrt{t}}f|^{2}\right\|^{\frac{1}{2}}_{\infty}+\frac{1}{t^{\frac{\alpha}{2}}}\left\|2\int_{0}^{t}T_{t-s}\Gamma( T_{s}P_{\sqrt{t}}f)ds\right\|^{\frac{1}{2}}_{\infty}
\\ \nonumber\lesssim &\sup_{t>0}\frac{1}{t^{\alpha}}\left\|P_{t}|f-P_{t}f|^{2}\right\|^{\frac{1}{2}}_{\infty}+\frac{1}{t^{\frac{\alpha}{2}}}\left\|2\int_{0}^{t}T_{t-s}\Gamma( T_{s}P_{\sqrt{t}}f)ds\right\|^{\frac{1}{2}}_{\infty}.
\end{align}
On the other hand, 
by \eqref{2} and the Fubini theorem, one gets
\begin{align*}
&\frac{1}{t^{\frac{\alpha}{2}}}\left\|2\int_{0}^{t}T_{t-s}\Gamma( T_{s}P_{\sqrt{t}}f)ds\right\|^{\frac{1}{2}}_{\infty}  \\ \nonumber=&\frac{1}{t^{\frac{\alpha}{2}}}\left\|2\int_{0}^{t}T_{t-s}\Gamma(\frac{1}{2\sqrt{\pi}}\int_{0}^{\infty}\sqrt{t}e^{-\frac{t}{4u}}u^{-\frac{3}{2}}T_{s+u}fdu)ds\right\|^{\frac{1}{2}}_{\infty}
\\ \nonumber \eqref{pm} \leq&\frac{1}{t^{\frac{\alpha}{2}}}\left\|\frac{1}{\sqrt{\pi}}\int_{0}^{\infty}\sqrt{t}e^{-\frac{t}{4u}}u^{-\frac{3}{2}}\int_{0}^{t}T_{t-s}\Gamma(T_{s+u}f)dsdu\right\|^{\frac{1}{2}}_{\infty}
\\ \nonumber(\Gamma^{2}\geq0) \leq&
\frac{1}{t^{\frac{\alpha}{2}}}\left\|\frac{1}{\sqrt{\pi}}\int_{0}^{\infty}\sqrt{t}e^{-\frac{t}{4u}}u^{-\frac{3}{2}}\int_{0}^{t}T_{t+u-\frac{t+u}{t}s}\Gamma(T_{\frac{t+u}{t}s}f)dsdu\right\|^{\frac{1}{2}}_{\infty} \\ \nonumber(v=\frac{t+u}{t}s)
\leq&\left(\frac{1}{\sqrt{\pi}}\int_{0}^{\infty}\sqrt{t}e^{-\frac{t}{4u}}u^{-\frac{3}{2}}\frac{t}{t+u}\frac{1}{t^{\alpha}}\left\|\int_{0}^{t+u}T_{t+u-v}\Gamma(T_{v}f)dv\right\|_{\infty}du\right)^{\frac{1}{2}}
\\ \nonumber(\text{Lemma}\,\, \ref{e.1}\,\, (\rm{i}))
=&\left(\frac{1}{2\sqrt{\pi}}\int_{0}^{\infty}\sqrt{t}e^{-\frac{t}{4u}}u^{-\frac{3}{2}}\frac{t^{1-\alpha}}{(t+u)^{1-\alpha}}\frac{1}{(t+u)^{\alpha}}\left\|T_{t+u}|f|^{2}-\left|T_{t+u} f\right|^{2}\right\|_{\infty}\right)^{\frac{1}{2}}
\\ \nonumber\leq&\left(\frac{1}{2\sqrt{\pi}}\int_{0}^{\infty}\sqrt{t}e^{-\frac{t}{4u}}u^{-\frac{3}{2}}\frac{t^{1-\alpha}}{(t+u)^{1-\alpha}}du\right)^{\frac{1}{2}}\sup_{t>0}\frac{1}{t^{\frac{\alpha}{2}}}\left\|T_{t}|f|^{2}-\left|T_{t} f\right|^{2}\right\|^{\frac{1}{2}}_{\infty}
\\ \nonumber (v=\frac{t}{4u})=&\left(\frac{1}{\sqrt{\pi}}\int_{0}^{\infty}e^{-v}v^{-\frac{1}{2}}(\frac{4v}{1+4v})^{1-\alpha}dv\right)^{\frac{1}{2}}\sup_{t>0}\frac{1}{t^{\frac{\alpha}{2}}}\left\|T_{t}|f|^{2}-\left|T_{t} f\right|^{2}\right\|^{\frac{1}{2}}_{\infty}.
\end{align*}
Let $h(\alpha)$ denote the function
\begin{align*}
h(\alpha)=\frac{1}{\sqrt{\pi}}\int_{0}^{\infty}e^{-v}v^{-\frac{1}{2}}(\frac{4v}{1+4v})^{1-\alpha}dv,\,\,0<\alpha<1.
\end{align*}
 It is easy to see that $h$ is increasing with respect to $0<\alpha<1$.
Then we get $h(\alpha)<h(1)=\frac{1}{\sqrt{\pi}}\int_{0}^{\infty}e^{-v}v^{-\frac{1}{2}}dv=1$.
Hence, \eqref{kas} can be estimated as
\begin{align*}
  \frac{1}{t^{\frac{\alpha}{2}}}\left\|T_{t}|f|^{2}-\left|T_{t} f\right|^{2}\right\|^{\frac{1}{2}}_{\infty} \leq\sup_{t>0}\frac{1}{t^{\alpha}}\left\|P_{t}|f-P_{t}f|^{2}\right\|^{\frac{1}{2}}_{\infty}+(h(\alpha))^{\frac{1}{2}}\sup_{t>0}\frac{1}{t^{\frac{\alpha}{2}}}\left\|T_{t}|f|^{2}-\left|T_{t} f\right|^{2}\right\|^{\frac{1}{2}}_{\infty}.
\end{align*}
Taking the supremum over $t>0$ on the left-hand side of the above inequality, we obtain
\begin{align*}
  \sup_{t>0}\frac{1}{t^{\frac{\alpha}{2}}}\left\|T_{t}|f|^{2}-\left|T_{t} f\right|^{2}\right\|^{\frac{1}{2}}_{\infty}
  \leq&\sup_{t>0}\frac{1}{t^{\alpha}}\left\|P_{t}|f-P_{t}f|^{2}\right\|^{\frac{1}{2}}_{\infty}+(h(\alpha))^{\frac{1}{2}}\sup_{t>0}\frac{1}{t^{\frac{\alpha}{2}}}\left\|T_{t}|f|^{2}-\left|T_{t} f\right|^{2}\right\|^{\frac{1}{2}}_{\infty}.
\end{align*}
Since $h(\alpha)<1$, we then have
\begin{align*}
\sup_{t>0}\frac{1}{t^{\frac{\alpha}{2}}}\left\|T_{t}|f|^{2}-\left|T_{t} f\right|^{2}\right\|^{\frac{1}{2}}_{\infty}  \leq\frac{1}{1-(h(\alpha))^{\frac{1}{2}}}\sup_{t>0}\frac{1}{t^{\alpha}}\left\|P_{t}|f-P_{t}f|^{2}\right\|^{\frac{1}{2}}_{\infty},
\end{align*}
which implies
\begin{align*}
 \|f\|_{\ell_{\frac{\alpha}{2}}^{c}(\mathcal T)}\lesssim\|f\|_{\ell_{\alpha}^{c}(\mathcal P)}.
\end{align*}
\end{proof}

\subsection*{Acknowledgments} 
The authors would like to thank the referee for his comments which help improve the presentation of the present paper.
This work is partially supported by the National
Natural Science Foundation of China (No. 12071355,
No. 12325105, No. 12031004, No. W2441002).

\bigskip

\noindent

\medskip
\noindent
Guixiang Hong\\
%
Institute for Advanced Study in Mathematics,
Harbin Institute of Technology,
Harbin 150001, China

\noindent

\noindent{E-mail address}:
\texttt{gxhong@hit.edu.cn} \\

\medskip
\noindent
Yuanyuan Jing\\
School of Mathematics and Statistics,
Wuhan University,
Wuhan 430072, China
\noindent

\noindent

\noindent{E-mail address}:
\texttt{yuanyuanjing@whu.edu.cn}\\
 \medskip

\noindent

\noindent

\end{document}